\documentclass[arxiv]{agn_article}

\usepackage[markexamples=true,markremarks=true]{agn_all}
\usepackage{caption}
\usepackage{multicol}
\usepackage{xstring}

\newcommand{\sample}{50}
\newcommand{\graphRatio}{.63}
\newcommand{\singleGraph}{.8}

\renewcommand{\Rn}{\R^n}
\DeclareMathOperator{\CSO}{CSO}

\DeclareMathOperator{\adj}{adj}
\newcommand{\lmax}{\lambdamax}
\newcommand{\lmin}{\lambdamin}
\newcommand{\Wmm}{W_{\mathrm{magic}}^-}
\newcommand{\Wmp}{W_{\mathrm{magic}}^+}

\newcommand{\Ws}{W_\mathrm{smooth}}
\newcommand{\Wn}{W_\mathrm{3b}}
\renewcommand{\Wlin}{W_\textrm{lin}}
\newcommand{\Wiso}{W_{\mathrm{iso}}}
\newcommand{\Wvol}{W_{\mathrm{vol}}}
\newcommand{\Winc}{W_{\mathrm{inc}}}
\newcommand{\sigmalin}{\sigma^\mathrm{lin}}
\newcommand{\Wliniso}{W_\mathrm{iso}^\mathrm{lin}}
\newcommand{\Wlinvol}{W_\mathrm{vol}^\mathrm{lin}}
\newcommand{\Mm}{\mathfrak{M}_-}
\newcommand{\Mp}{\mathfrak{M}_+}
\newcommand{\Mms}{\mathfrak{M}_-^*}
\newcommand{\Mps}{\mathfrak{M}_+^*}

\newcommand{\ddp}{\frac{\mathrm{d}}{\mathrm{d}p}}
\newcommand{\ddr}{\frac{\mathrm{d}}{\mathrm{d}r}}
\newcommand{\dth}{\mathrm{d}\vartheta}
\newcommand{\casesif}{\caseif}

\newcommand{\BRz}{B_R(0)}

\usepackage{pifont}
\newcommand{\xmark}{\ding{55}}

\providecommand{\citet}[2][]{\citeauthor{#2} \cite[#1]{#2}}

\tikzset{
    state/.style={
           rectangle,
           rounded corners,
           draw=black, very thick,
           minimum height=2em,
           inner sep=6pt,
           text centered,
           }
}
\begin{document}
\title{Morrey's conjecture for the planar volumetric-isochoric split.\\ \mbox{Part I:} least convex energy functions}
\date{\today}
\knownauthors[voss]{voss,martin,ghiba,neff}
\maketitle
\vspace{-1.5em}
\begin{abstract}
\noindent
We consider Morrey's open question whether rank-one convexity already implies quasiconvexity in the planar case. For some specific families of energies, there are precise conditions known under which rank-one convexity even implies polyconvexity. We will extend some of these findings to the more general family of energies $W:\operatorname{GL}^+(n)\rightarrow\mathbb{R}$ with an additive volumetric-isochoric split, i.e.
\[
	W(F)=W_{\rm iso}(F)+W_{\rm vol}(\det F)=\widetilde W_{\rm iso}\bigg(\frac{F}{\sqrt{\det F}}\bigg)+W_{\rm vol}(\det F)\,,
\]
which is the natural finite extension of isotropic linear elasticity. Our approach is based on a condition for rank-one convexity which was recently derived from the classical two-dimensional criterion by Knowles and Sternberg and consists of a family of one-dimensional coupled differential inequalities. We identify a number of \enquote{least} rank-one convex energies and, in particular, show that for planar volumetric-isochorically split energies with a concave volumetric part, the question of whether rank-one convexity implies quasiconvexity can be reduced to the open question of whether the rank-one convex energy function
\[
	W_{\rm magic}^+(F)=\frac{\lambda_{\rm max}}{\lambda_{\rm min}}-\log\frac{\lambda_{\rm max}}{\lambda_{\rm min}}+\log\det F=\frac{\lambda_{\rm max}}{\lambda_{\rm min}}-2\log\lambda_{\rm min}
\]
is quasiconvex. In addition, we demonstrate that under affine boundary conditions, $W_{\rm magic}^+(F)$ allows for non-trivial inhomogeneous deformations with the same energy level as the homogeneous solution,
and show a surprising connection to the work of Burkholder and Iwaniec in the field of complex analysis. 
\end{abstract}
\\[.65em]

{\parskip=-0.4mm \tableofcontents}
%
%
%
%
\vspace{-0.6em}
\section{Introduction}
One of the major open problems in the calculus of variations is Morrey's problem in the planar case. The open question is whether or not rank-one convexity implies quasiconvexity. Morrey \cite{morrey1952quasi,morrey2009multiple} conjectured that this is not the case in general and indeed the result by Sverak \cite{vsverak1992rank} settles the question for the case $n\geq 3$. On the other hand, in the one-dimensional case, all four notions of convexity are equivalent. For the last nearly thirty years, impressive progress on the problem has been made without, however, leading to a breakthrough. While researchers coming from the area of nonlinear elasticity in general seem to believe that Morrey's conjecture is true also in the planar case, i.e.\ that rank-one convexity does not imply quasiconvexity, researchers with a background in conformal and quasiconformal analysis \cite{astala2008elliptic} tend to believe in the truth of the opposite implication \cite{iwaniec1982extremal}. In conformal analysis, this conclusion would imply a number of far-reaching conjectures which are themselves deemed reasonable at present.

Nonlinear elasticity is concerned with energies $W$ defined on the group of matrices with positive determinant $\GLp(n)=\left\{F\in\Rnn|\det F>0\right\}$, while the general case of Morrey's question considers scalar functions $W$ defined on the linear space $\Rnn$. It is conceivable that the answer to the planar Morrey's conjecture depends on whether one focuses on $\GLp(2)$ or $\R^{2\times 2}$. Assuming that an example for Morrey's conjecture in $\GLp(2)$ can be found, it is also not immediately clear how such an example might be generalized to the $\R^{2\times 2}$-case (although it is tempting to believe that this is always possible). The precise statement of Morrey's conjecture also depends on the invariance properties one requires for the energy $W$ and the domain of definition. Different versions of Morrey's conjecture include the following, all of which currently remain open:
{%
\setlength{\leftmargini}{7em}%
\newcommand{\morreyindent}[1]{\rlap{#1}\hphantom{Morrey III:}}
\begin{itemize}
	\item[\morreyindent{Morrey I:}] Let $W\col\Rnn\to\R$ (with no further invariance requirement). If $W$ is rank-one convex, then $W$ is not necessarily quasiconvex. For $n\geq 3$ ($W$ being continuous), this statement has been conclusively settled by Sverak's famous counterexample \cite{vsverak1992rank}.
	\item[\morreyindent{Morrey II:}] Let $W\col\R^{2\times 2}\to\R$ be continuous and left/right-$\SO(2)$ invariant (objective and hemitropic). If $W$ is rank-one convex, then $W$ is not necessarily quasiconvex.
	\item[\morreyindent{Morrey III:}] Let $W\col\GLp(2)\to\R\cup\{\infty\}$ and $W(F)\to\infty$ as $\det F\to 0$ be left-$\SO(2)$ invariant (objective, frame-indifferent) and right-$\OO(2)$ invariant (isotropic). If $W$ is rank-one convex, then $W$ is not necessarily quasiconvex.
	\item[\morreyindent{Morrey IV:}] Let $W\col\R^{2\times 2}\to\R$  be left-$\SO(2)$ invariant (objective, frame-indifferent) and right-$\OO(2)$ invariant (isotropic). If $W$ is rank-one convex, then $W$ is not necessarily quasiconvex.
	\item[\morreyindent{Morrey V:}] Let $W\col\GLp(2)\to\R\cup\{\infty\}$ and $W(F)\to\infty$ as $\det F\to 0$ be left/right-$\SO(2)$ invariant (objective and hemitropic). If $W$ is rank-one convex, then $W$ is not necessarily quasiconvex.
\end{itemize}}
Iwaniec, Astala \cite{astala2012quasiconformal} and their groups try to disprove Morrey II to prove various conjectures in the context of complex analysis. It is conceivable, however, that Morrey II is false, while Morrey III is true. Similarly, Morrey III may also be false with Morrey II being true. There is no clear connection between the two statements: in Morrey III the domain of definition is smaller, but then rank-one convexity needs only to hold true on a smaller set as compared to Morrey II. The role of invariance requirements (isotropic versus hemitropic) or general anisotropy is also open.

The restriction to $\GLp(2)$ and isotropy in Morrey III is advantageous because the check of rank-one convexity can be based on a version of Knowles/Sternberg conditions \cite{knowles1976failure,knowles1978failure} operating only on the representation of the elastic energy in terms of singular values. The rank-one convexity condition on $\R^{2\times 2}$ was also completely settled by \citet{hamburger1987characterization,aubert1995necessary}. Nevertheless, it seems that one has to choose a favorable representation first and then use an adapted criterion. Quasiconvexity on the other hand is notoriously difficult to prove or disprove. A straightforward sufficient condition is Sir John Ball's polyconvexity condition \cite{ball1976convexity}. One of the advantages of polyconvexity is that it is both suitable for $\GLp(2)$ and $\R^{2\times 2}$ and that there are sharp characterizations of polyconvexity in the $\GLp(2)$ and isotropic case \cite{Dacorogna08,silhavy1997mechanics,mielke2005necessary}.

Apparently, even for $n\geq3$, there is no isotropic example to Morrey I to be found in the literature. A common method to approach Morrey's conjecture is to discuss a specific family of energies, e.g.\ purely volumetric or isochoric energies (cf.\ Section \ref{sec:purelyVolumetric} and \ref{sec:purelyIsochoric}), to find an often simple relation between the different notions of convexity. We will extend these findings to a much more involved class of energies with an additive volumetric-isochoric split \eqref{eq:volIsoSplit} and follow the concept of \enquote{least} rank-one convex energies, i.e.\ functions that are as weakly rank-one convex as possible, to serve as canonical candidates. Focusing on energies with a concave volumetric part we show that the rank-one convex function
\[
	W_{\rm magic}^+(F)=\frac{\lambdamax}{\lambdamin}-\log\frac{\lambdamax}{\lambdamin}+\log\det F=\frac{\lambdamax}{\lambdamin}-2\.\log\lambdamin
\]
with $\lmax\geq\lmin>0$ as the ordered singular values of the deformation gradient $F=\nabla\varphi$ suffices as the sole candidate for this class of functions. The question of quasiconvexity for $\Wmp(F)$ closes Morrey's conjecture for all volumetric-isochorically split energies with a concave volumetric part (cf.\ Lemma \ref{lem:minimallyConvexEnergy}).

In a sequel of this work \cite{agn_voss2021morrey2} we will numerically investigate the quasiconvexity of $\Wmp(F)$.
%
%
%
%
\subsection{Basic definitions of convexity properties}\label{sec:convexity}
\begin{definition}\label{def:quasiconvexity}
	The energy function $W\col\Rnn\to\R\cup\{+\infty\}$ is \emph{quasiconvex} if and only if
	\begin{equation}
		\int_\Omega W(F_0+\nabla\vartheta(x))\,\dx\geq\int_\Omega W(F_0)\,\dx=\abs\Omega\cdot W(F_0)\qquad\text{for all }F_0\in\Rnn,\;\vartheta\subset W_0^{1,\infty}(\Omega;\Rn)\label{eq:quasiconvexity}
	\end{equation}
	for any domain $\Omega\subset\Rn$ with Lebesgue measure $\abs{\Omega}$. The energy is \emph{strictly quasiconvex} if inequality \eqref{eq:quasiconvexity} is strict.
\end{definition}
While quasiconvexity is sufficient to ensure weak lower semi-continuity of the energy functional, it has the simple disadvantage of being almost impossible to prove or disprove. This led to the introduction of sufficient and necessary conditions for quasiconvexity. 
\begin{definition}\label{def:polyconvexity}
	The energy function $W\col\Rnn\to\R\cup\{+\infty\}$ is \emph{polyconvex} if and only if there exists a convex function $P\col\R^{m(n)}\to\R\cup\{+\infty\}$ with
	\begin{equation}
		W(F)=P(F,\adj_2(F),\cdots,\adj_n(F))\qquad\text{for all}\quad F\in\Rnn\label{eq:polyconvex}
	\end{equation}
	with $\adj_i(F)\in\Rnn$ as the matrix of the determinants of all $i\times i-$minors of $F$ and $m(n) \colonequals \sum_{i=1}^n \binom{n}{i}^2$. The energy is \emph{strictly polyconvex} if there exists a representation $P$ which is strictly convex.	
\end{definition}
In particular, for the planar case $n=2$, the energy $W\col\R^{2\times 2}\to\R\cup\{+\infty\}$ is polyconvex if and only if there exists a convex mapping $P\col\R^{2\times 2}\times\R\cong\R^5\to\R\cup\{+\infty\}$ with
\[
	W(F)=P(F,\det F)\qquad\text{for all}\quad F\in\R^{2\times 2}\,.
\]
Whereas polyconvexity provides a sufficient criterion for quasiconvexity, rank-one convexity is introduced as a necessary condition.
\begin{definition}\label{def:rank1convexity}
	The energy $W\col\Rnn\to\R\cup\{+\infty\}$ is \emph{rank-one convex} if and only if
	\begin{equation}
		W(t\.F_1+(1-t)\.F_2)\leq t\.W(F_1)+(1-t)\.W(F_2)\qquad\text{for}\quad\rank(F_1-F_2)=1\,,\quad t\in(0,1)\,.\label{eq:rankOneConvexity}
	\end{equation}
	If the energy is two-times differentiable, global rank-one convexity is equivalent to the \emph{Legendre-Hadamard ellipticity condition}
	\begin{equation}
		D^2W(F).(\xi\otimes\eta,\xi\otimes\eta)\geq 0 \qquad\text{for all}\quad F\in\Rnn,\;\xi,\eta\in\Rn\label{LegendreHadamardEllipticity}
	\end{equation}
	which expresses the ellipticity of the Euler-Lagrange equations $\Div DW(\nabla\varphi)=0$ associated to the variational problem
	\begin{equation}
		I(\varphi)=\int_\Omega W(\nabla\varphi(x))\,\dx\to\min\,.\label{eq:minimizationProblem}
	\end{equation}
	The energy is \emph{strictly rank-one convex} if inequality \eqref{eq:rankOneConvexity} is strict.
\end{definition}

Overall, for any $W\col\Rnn\to\R\cup\{+\infty\}$ we have the well known hierarchy \cite{ball1976convexity,Dacorogna08}
\begin{equation}
	\text{polyconvexity}\quad\implies\quad\text{quasiconvexity}\quad\implies\quad\text{rank-one convexity}.
\end{equation}

Energy functions in nonlinear elasticity are naturally defined on $\GLpn$ instead of $\Rnn$ to prevent self-intersection. Thus in the special case of functions defined on the domain $\GLp(n)$ we introduce:
\begin{definition}
	A function $W\col\GLpn\to\R$ is called quasi-/poly-/rank-one convex if the function
	\[
		\What\col\Rnn\to\R\cup\{+\infty\}\,,\quad\What(F)=\begin{cases}
			W(F) &\casesif F\in\GLpn\,,\\
			+\infty &\casesif F\notin\GLpn\,,
		\end{cases}
	\]
is quasi-/poly-/rank-one convex. A function $W\col\GLpn\to\R$ is called convex if there exists a convex function $\What\col\Rnn\to\R$ such that $\What(F)=W(F)$ for all $F\in\GLpn$.
\end{definition}
%
%
%
%
%
\section{The additive volumetric-isochoric split}\label{sec:volisoSplit}
Next, we discuss the general family of planar energies $W\col\GLp(2)\to\R\cup\{\infty\}$ with additive combinations of conformally invariant (isochoric) energies on the one hand, and area-dependent (volumetric) energies (also called the \emph{additive volumetric-isochoric split}) on the other hand:
\begin{align}
	W(F)=\Wiso(F)+\Wvol(\det F)=\underbrace{\widetilde W_\textrm{iso}\bigg(\frac{F}{\sqrt{\det F}}\bigg)}_{\mathclap{\text{conformally invariant}}}+\underbrace{\Wvol(\det F)}_{\hspace{2cm}\mathclap{\text{purely area-dependent in 2D}}}.\label{eq:volIsoSplit}
\end{align}
It is already known that for purely isochoric energies (cf.\ Section \ref{sec:purelyIsochoric}) as well as purely volumetric energies (cf.\ Section \ref{sec:purelyVolumetric}) rank-one convexity already implies polyconvexity (and thereby quasiconvexity). Thus, it is suggestive to inquire whether the same conclusion can be reached for additive combinations or not.

The additive volumetric-isochoric split itself is a natural finite extension of isotropic linear elasticity, i.e.\
\begin{align}
	\Wlin(\nabla u)&=\mu\.\norm{\sym\nabla u}^2+\frac{\lambda}{2}\.(\tr\nabla u)^2=\mu\.\norm{\dev_n\sym\nabla u}^2+\frac{\kappa}{2}\.(\tr\nabla u)^2\,,\label{eq:linEla}\\
	\sigmalin&=D\Wlin(\nabla u)=2\mu\.\dev_n\sym\nabla u+\kappa\.\tr(\nabla u)\.\id\,.\notag
\end{align}
The right hand side of the energy is automatically additively separated into pure infinitesimal shape change and infinitesimal volume change of the displacement gradient $\nabla u$, with a similar additive split of the linear Cauchy stress tensor into a deviatoric part and a spherical part, depending only on the shape change and volumetric change, respectively.

This type of energy functions is often used for modeling slightly incompressible material behavior \cite{Ciarlet1988,agn_hartmann2003polyconvexity,favrie2014thermodynamically,ogden1978nearly,charrier1988existence} or when otherwise no detailed information on the actual response of the material is available \cite{simo1988framework,gavrilyuk2016example,ndanou2017piston}. In the nonlinear regime, this split is not a fundamental law of nature for isotropic bodies (as it is in the linear case) but rather introduces a convenient form of the representation of the energy. Formally, this split can also be generalized to the anisotropic case, where it shows, however, some severe deficiencies \cite{federico2010volumetric,murphy2018modelling} from a modeling point of view.\footnote{For example, a perfect sphere made of an \emph{anisotropic} material and subject only to all around uniform pressure would remain spherical for volumetric-isochorically decoupled energies.} 

A primary application of the volumetric-isochoric split is the adaptation of \emph{incompressible} hyperelastic models to (slightly) compressible materials in arbitrary dimension $n\geq2$: For any given isotropic function\footnote{Here, $\SLn$ denotes the \emph{special linear group}, i.e.\ $\SLn=\left\{X\in\Rnn\,|\,\det X=1\right\}.$} $\Winc\col\SLn\to\R$, an elastic energy potential $W$ of the form \eqref{eq:volIsoSplit} can be constructed by simply setting
\[
	W(F) = \Winc\biggl(\frac{F}{\sqrt[n]{\det F}}\biggr) + f(\det F)
	= \Winc(\Ihat_1,\dotsc,\Ihat_{n-1}) + f(\det F)
\]
with an appropriate function $f\col\Rp\to\R$, where $\Ihat_1,\dotsc,\Ihat_{n-1}$ denote the principal invariants of $\frac{F}{\sqrt[n]{\det F}}$. Common examples of such energy functions include the \emph{compressible Neo-Hooke} and the \emph{compressible Mooney-Rivlin} model \cite{Ogden83}.

A typical example of the volumetric-isochoric format is the geometrically nonlinear quadratic Hencky energy \cite{agn_neff2015geometry,Hencky1928,sendova2005strong,agn_martin2018polyconvex} 
\begin{align}
	W_{\rm H}(F)&=\mu\.\norm{\dev_n\log V}^2+\frac{\kappa}{2}\left(\tr\log V\right)^2,\qquad V=\sqrt{F\.F^T},\label{eq:hencky}\\
	\tau_{\rm H}&=2\mu\.\dev_n\log V+\kappa\.\tr(\log V)\cdot\id=2\mu\.\log\frac{V}{\sqrt[3]{\det V}}+\kappa\.\log\det F\cdot\id\notag
\end{align}
as well as its physically nonlinear extension, the exponentiated Hencky-model \cite{agn_neff2015exponentiatedI,agn_neff2014exponentiatedIII}
\begin{align}
	W_{\rm eH}&=\frac{\mu}{k}\.e^{k\.\norm{\dev_n\log V}^2}+\frac{\kappa}{2\khat}\.e^{\khat\left(\tr\log V\right)^2}\,,\label{eq:expHencky}\\
	\tau_{\rm eH}&=2\mu\.e^{k\.\|\dev_n\log\,V\|^2}\cdot\dev_n\log V+\underbrace{\kappa\.e^{\khat\left(\tr\log V\right)^2}\.\tr(\log V)\cdot\id}_{=\kappa\.e^{\khat\left(\log\det F\right)^2}\.\log\det F\cdot\id}\,,\notag
\end{align}
which has been used for stable computation of the inversion of tubes \cite{agn_nedjar2018finite}, the modeling of tire-derived material \cite{agn_montella2015exponentiated} or applications in bio-mechanics \cite{agn_schroder2018exponentiated}. It is well known that \eqref{eq:hencky} and \eqref{eq:expHencky} are overall not rank-one convex \cite{agn_neff2015exponentiatedI} and indeed there does not exist any elastic energy depending on $\norm{\dev_n\log V}^2$, $n\geq 3$ that is rank-one convex \cite{agn_ghiba2015ellipticity}. The situation is, surprisingly, completely different in the planar case: For $n=2$, the energy in equation \eqref{eq:expHencky} is not only rank-one convex, but even polyconvex \cite{agn_neff2015exponentiatedII,agn_ghiba2015exponentiated} if $k\geq\frac{1}{8}$.
%
%
%
%
%
\subsection{Purely volumetric energies}\label{sec:purelyVolumetric}
For some types of energies $W\col\R^{2\times 2}\to\R$, precise conditions are known under which rank-one convexity implies polyconvexity and therefore quasiconvexity \cite{muller1999rank,ball2002openProblems}. One specific family of energies for which the relation between rank-one convexity and polyconvexity is already known is that of purely volumetric energies $W(F)=f(\det F)$. This class describes fluid-like material behavior, i.e.\ shape change is not penalized at all.
\begin{lemma}[\cite{Dacorogna08}]\label{lemma:volumetricConvexity}
	Let $W\col\GLp(2)\to\R$ be of the form $W(F)=f(\det F)$ with $f\col\Rp\to\R$. The following are equivalent:
	\setlength\columnsep{-10em}
	\begin{multicols}{2}
		\begin{itemize}
			\item[i)] $W$ is polyconvex,
			\item[ii)] $W$ is quasiconvex,
			\item[iii)] $W$ is rank-one convex,
			\item[iv)] $f$ is convex on $(0,\infty)$.
		\end{itemize}
	\end{multicols}
\end{lemma}
The same holds for the so-called \emph{Hadamard-material} \cite{Dacorogna08,ball1984w1,grabovsky2018rank,dantas2006equivalence}:
\begin{lemma}[\cite{Dacorogna08}]\label{lemma:hadamrdConvexity}
	Let $W\col\GLp(n)\to\R$ of the form $W(F)=\norm{F}^\alpha+f(\det F)$ with $f\col\Rp\to\R$ and $1\leq\alpha<2\.n$. The following are equivalent:
	\setlength\columnsep{-10em}
	\begin{multicols}{2}
		\begin{itemize}
			\item[i)] $W$ is polyconvex,
			\item[ii)] $W$ is quasiconvex,
			\item[iii)] $W$ is rank-one convex,
			\item[iv)] $f$ is convex on $(0,\infty)$.
		\end{itemize}
	\end{multicols}
\end{lemma}
%
%
%
%
%
\subsection{Purely isochoric energies}\label{sec:purelyIsochoric}
In a recent contribution \cite{agn_martin2015rank}, we present a similar statement for isochoric energies, also called \emph{conformally invariant energies}, i.e.\ functions $W\col\GLp(2)\to\R$ with
\begin{equation}\label{eq:definitionConformalInvariance}
	W(A\.F\.B) = W(F) \qquad\text{for all}\quad A,B\in\CSO(2)\,,
\end{equation}
where
\begin{equation}
	\CSO(2)\colonequals\R^+\cdot\SO(2)=\{a\.R\in\GLp(2) \setvert a\in(0,\infty)\,,\; R\in\SO(2)\}\label{eq:CSO}
\end{equation}
denotes the \emph{conformal special orthogonal group}. This requirement can equivalently be expressed as
\begin{equation}\label{eq:objectiveIsotropicIsochoric}
	W(R_1F)=W(F)=W(FR_2)\,,\quad W(aF)=W(F) \qquad\text{for all}\quad R_1,R_2\in\SO(2)\,,\; a\in(0,\infty)\,,
\end{equation}
i.e.\ left- and right-invariance under the special orthogonal group $\SO(2)$ and invariance under scaling. An energy $W$ satisfying $W(aF)=W(F)$ is more commonly known as \emph{isochoric}. Following Iwaniec et al. \cite{astala2008elliptic}, we introduce the \emph{(nonlinear) distortion function} or \emph{outer distortion}
\begin{align}
	\K\col\GLp(2)\to\R\,,\qquad\K(F)\colonequals\frac{1}{2}\.\frac{\norm F^2}{\det F}=\frac{\lambda_1^2+\lambda_2^2}{2\.\lambda_1\lambda_2}=\frac{1}{2}\left(\frac{\lambda_1}{\lambda_2}+\frac{\lambda_2}{\lambda_1}\right),\label{eq:distortion}
\end{align}
where $\norm{\,.\,}$ denotes the Frobenius matrix norm with $\norm{F}^2=\sum_{i,j=1}^2 F_{ij}^2$ as well as the \emph{linear distortion} or \emph{(large) dilatation}
\begin{align}
	K(F)\col\GLp(2)\to\R\,,\qquad K(F)=\frac{\opnorm{F}^2}{\det F}=\frac{\lambdamax^2}{\lambdamin\lambdamax}=\frac{\lambdamax}{\lambdamin}\,,\label{eq:dilatation}
\end{align}
where $\opnorm{F}=\sup_{\norm{\xi}=1}\norm{F\.\xi}_{\R^2}=\lmax$ denotes the operator norm (i.e.\ the largest singular value) of $F$. 

\begin{lemma}[\cite{agn_martin2015rank}]
\label{lemma:representations}
	Let $W\col\GLp(2)\to\R$ be conformally invariant. Then there exist uniquely determined functions $g\col\Rp\times\Rp\to\R$,\; $h\col\Rp\to\R$,\; $\hhat\col[1,\infty)\to\R$ and $\Psi\col[1,\infty)\to\R$ such that
	\begin{equation}\label{eq:representationFormulae}
		W(F) = g(\lambda_1,\lambda_2) = h\left(\frac{\lambda_1}{\lambda_2}\right) = \hhat(K(F)) = \Psi(\K(F))
	\end{equation}
	for all $F\in\GLp(2)$ with (not necessarily ordered) singular values $\lambda_1,\lambda_2$. Furthermore,
	\begin{equation}\label{eq:representationRequirements}
		h(t)=h\left(\frac1t\right),\quad g(x,y)=g(y,x) \quad\text{ and }\quad g(ax,ay)=g(x,y)
	\end{equation}
	for all $a,t,x,y\in(0,\infty)$.
\end{lemma}
Note that the isotropy requirement \eqref{eq:representationRequirements} implies that $h$ is already uniquely determined by its values on $[1,\infty)$. While we do distinguish between $\lambda_1,\lambda_2$ as the non-ordered singular values and \mbox{$\lmax\geq\lmin$} as the ordered singular values, we will interchange $h$ and $\hhat$ to simplify some notations here because $h(t)=\hhat(t)$ for all $t\geq 1$.
\begin{lemma}[\cite{agn_martin2015rank}]\label{lemma:convexityCharacterization}
	Let $W\col\GLp(2)\to\R$ be conformally invariant,
	and let $g\col\Rp\times\Rp\to\R$,\; $h\col\Rp\to\R$ and $\Psi\col[1,\infty)\to\R$ denote the uniquely determined functions with
	\[
		W(F) = g(\lambda_1,\lambda_2) = h\left(\frac{\lambda_1}{\lambda_2}\right) = \Psi(\K(F))
	\]
	for all $F\in\GLp(2)$ with singular values $\lambda_1,\lambda_2$, where $\K(F)=\frac12\frac{\norm{F}^2}{\det F}$. Then the following are equivalent:
	\setlength\columnsep{-7em}
	\begin{multicols}{2}
		\begin{itemize}
			\item[i)] $W$ is polyconvex,
			\item[ii)] $W$ is quasiconvex,
			\item[iii)] $W$ is rank-one convex,
			\item[iv)] $g$ is separately convex,
			\item[v)] $h$ is convex on $(0,\infty)$,
			\item[vi)] $h$ is convex and non-decreasing on $[1,\infty)$.
		\end{itemize}
	\end{multicols}
\end{lemma}
%
%
%
%
%
\section{Least rank-one convex candidates}\label{sec:leastRankOne}
We aim to discuss Morrey's conjecture in this setting, i.e.:
\begin{itemize}
	\item Is every rank-one convex energy with an additive volumetric-isochoric split in 2D already \mbox{quasiconvex?}
\end{itemize}
For this purpose we introduce the concept of \enquote{least} rank-one convex functions: If an energy function $W_0(F)=W_1(F)+W_2(F)$ can be written as the sum of two individual rank-one convex energies, it is sufficient to test $W_1,W_2$ instead of $W_0$ for quasiconvexity; If $W_0$ would be an example for Morrey's conjecture, i.e.\ not quasiconvex, at least one of the two summands $W_1$ or $W_2$ has to be not quasiconvex as well, since the sum of two individual quasiconvex terms remains quasiconvex. Consequently we will only test those rank-one convex functions for quasiconvexity which cannot be split into a convex combination of individual rank-one convex functions without considering the set of \emph{Null-Lagrangians}. In other words, we are only interested in the verticies of the set of rank-one convex energy function and call these functions \emph{least rank-one convex candidates}. If there exists a rank-one convex energy function which is not quasiconvex, it must exist a least rank-one convex function which is not quasiconvex as well.

Therefore, the main goal of this section is to reduce the class of energies with an additive volumetric-isochoric split to a set of least rank-one convex representatives as candidates to check for quasiconvexity. We identify these least rank-one convex energies by evaluating the set of coupled one-dimensional inequalities from Theorem \ref{theorem:mainVolisLog2} and looking for non-strictness. An overview of this section is given on page \pageref{fig:morreySummary}.
%
%
%
%
\subsection{Rank-one convexity}
For objective and isotropic planar energies with additive volumetric-isochoric split there exists a unique representation \cite[Lemma 3.1]{agn_martin2015rank}
\begin{align}
	W(F)=h\bigg(\frac{\lambda_1}{\lambda_2}\bigg)+f(\lambda_1\lambda_2)\,,\qquad h,f\col\Rp\to\R\,,\quad h(t)=h\bigg(\frac{1}{t}\bigg)\qquad\text{for all}\quad t\in(0,\infty)\,,\label{eq:definitionOfh}
\end{align}
where $\lambda_1,\lambda_2>0$ denote the singular values of $F$ and $h,f$ are given real-valued functions. In view of Lemma \ref{lemma:hadamrdConvexity} and \ref{lemma:convexityCharacterization} it is now tempting to believe that rank-one convexity conditions on $\Wiso$ and $\Wvol$ also allow for a sort of split and can be reduced to separate statements on $h$ and $f$. However, this is not even the case in planar linear elasticity, where $\Wlin$ from equation \eqref{eq:linEla} is rank-one convex in the displacement gradient $\nabla u$ if and only if $\mu\geq 0$ and $\mu+\kappa\geq 0$. This means that rank-one convexity of $\Wlin$ implies that $\Wliniso$ is rank-one convex, whereas $\Wlinvol$ might not be rank-one convex, since e.g.\ $\kappa<0$ is allowed. We are therefore prepared to expect some coupling in the conditions for $h$ and $f$. 

For rank-one convexity on $\GLp(2)$, Knowles and Sternberg \cite{knowles1976failure,knowles1978failure} (see also \cite[p.~308]{silhavy1997mechanics}) established the following important and useful criterion.
\begin{lemma}[{Knowles and Sternberg \cite{knowles1976failure,knowles1978failure}, cf.\  \cite{silhavy1997mechanics,parry1995planar,dacorogna01,vsilhavy2002convexity,vsilhavy1999isotropic,aubert1987faible,aubert1995necessary}}]\label{lemma:knowlesSternberg}
	Let $W\col\GLp(2)\to\R$ be an objective-isotropic function of class $C^2$ with the representation in terms of the singular values of the deformation gradient $F$ given by $W(F)=g(\lambda_1,\lambda_2)$, where $g\in C^2(\Rp^2,\R)$. Then $W$ is rank-one convex if and only if the following five conditions are satisfied:
	\begin{itemize}
		\setlength\itemsep{-4pt}
		\item[i)] \qquad$\displaystyle g_{xx}\geq 0 \text{ and } g_{yy}\geq 0\qquad \text{for all}\quad x,y\in (0,\infty)\,,$\hfill (separate convexity)
		\item[ii)] \qquad$\displaystyle\frac{x\.g_x-y\.g_y}{x-y}\geq 0\qquad\text{for all}\quad x,y\in (0,\infty)\,,\; x\neq y\,,$\hfill (Baker-Ericksen inequalities)
		\item[iii)] \qquad$\displaystyle g_{xx}-g_{xy}+\frac{g_x}{x}\geq 0\quad\text{and}\quad g_{yy}-g_{xy}+\frac{g_y}{y}\geq 0\qquad\text{for all}\quad x,y\in (0,\infty)\,,\; x=y\,,$
		\item[iv)]\qquad$\displaystyle\sqrt{g_{xx}\,g_{yy}}+g_{xy}+\frac{g_x-g_y}{x-y}\geq 0\qquad\text{for all}\quad x,y\in (0,\infty)\,,\; x\neq y\,,$
		\item[v)]\qquad$\displaystyle\sqrt{g_{xx}\,g_{yy}}-g_{xy}+\frac{g_x+g_y}{x+y}\geq 0\qquad\text{for all}\quad x,y\in (0,\infty)\,.$
	\end{itemize}
	Furthermore, if all the above inequalities are strict, then $W$ is strictly rank-one convex.	
\end{lemma}

In a recent contribution \cite{agn_voss2019volIsLog} we applied the criterion by Knowles and Sternberg to volumetric-isochorically split energies: using weakened regularity assumptions \cite{agn_martin2019regularity} the following conditions can be established.
\begin{theorem}[\cite{agn_voss2019volIsLog,agn_martin2019regularity}]\label{theorem:mainVolisLog2}
	Let $W\col\GLp(2)\to\R$ be an objective-isotropic function with 
	\[
		W(F)=\ghat(\lmax,\lmin)=\hhat\left(\frac{\lmax}{\lmin}\right)+f(\lmax\lmin)\qquad\text{for all}\quad F\in\GLp(2)
	\]
	and ordered singular values $\lmax\geq\lmin$, where $\hhat\in C^2((1,\infty);\R)$ and $f\in C^2(\Rp;\R)$. We write $h\col\Rp\to\R$ with $h(t)\colonequals\hhat(t)$ for $t\geq1$ and $h(t)\colonequals\hhat\bigl(\frac1t\bigr)$ for $t<1$. Let
	\[
		h_0=\inf_{t\in(1,\infty)}t^2\.h''(t)\qquad\text{and}\qquad f_0=\inf_{z\in(0,\infty)}z^2\.f''(z)\,.
	\]
	Then $W$ is rank-one convex if and only if
	\begin{itemize}
		\item[1)] \qquad$\displaystyle h_0+f_0\geq 0\,,$\qquad (separate convexity)
		\item[2)] \qquad$\displaystyle h'(t)\geq 0\,,$\qquad (Baker-Ericksen inequality)
		\item[3)]\qquad$\displaystyle \frac{2\.t}{t-1}\.h'(t)-t^2\.h''(t)+f_0\geq 0$\quad or\quad $\displaystyle a(t)+\left[b(t)-c(t)\right]f_0\geq 0\,,$
		\item[4)]\qquad$\displaystyle \frac{2\.t}{t+1}\.h'(t)+t^2\.h''(t)-f_0\geq 0$\quad or\quad $\displaystyle a(t)+\left[b(t)+c(t)\right]f_0\geq 0$
	\end{itemize}
	for all $t>1$, where
	\begin{align*}
		a(t)&=t^2(t^2-1)\.h'(t)h''(t)-2t\.h'(t)^2,\qquad b(t)=(t^2+3)\.h'(t)+2t\.(t^2+1)\.h''(t)\,,\\
		c(t)&=4t\.\bigl(h'(t)+t\.h''(t)\bigr).
	\end{align*}
\end{theorem}
%
%
%
%
\subsection{The volumetric part $f(z)$}
Theorem \ref{theorem:mainVolisLog2} yields:
\begin{lemma}\label{lemma:necessary}
	Let $W\col\GLp(2)\to\R$  be a rank-one convex isotropic energy function with an additive volumetric-isochoric split $W(F)=h\bigl(\frac{\lambda_1}{\lambda_2}\bigr)+f(\lambda_1\lambda_2)$. Then $h$ is convex on $(0,\infty)$ or $f$ is convex on $(0,\infty)$.
\end{lemma}
\begin{proof}
	The separate convexity condition 1) in Theorem \ref{theorem:mainVolisLog2} implies
	\begin{align}
		h_0+f_0\geq 0\qquad\implies\qquad &h_0\geq 0\quad \text{or}\quad f_0\geq 0\\
		\iff\qquad &t^2\.h''(t)\geq 0\quad\text{for all }t\in(1,\infty)\quad \text{or}\quad z^2\.f''(z)\geq 0\quad\text{for all }z\in(0,\infty)\notag\\
		\iff\qquad &h''(t)\geq 0\quad\text{for all }t\in(1,\infty)\quad \text{or}\quad f''(z)\geq 0\quad\text{for all }z\in(0,\infty)\,.\notag
	\end{align}
	The invariance property $h(t)=h\big(\frac{1}{t}\big)$ of the isochoric part $h$ yields
	\begin{align}
		h'(t)&=\ddt h(t)=\ddt h\left(\frac{1}{t}\right)=-\frac{1}{t^2}\.h'\left(\frac{1}{t}\right),\label{eq:symmetryofh1}\\
		h''(t)&=\ddt h'(t)=\ddt\left[-\frac{1}{t^2}\.h'\left(\frac{1}{t}\right)\right]=\frac{2}{t^3}\.h'\left(\frac{1}{t}\right)+\frac{1}{t^4}\.h''\left(\frac{1}{t}\right).\label{eq:symmetryofh}
	\end{align}
	Thus for arbitrary $s\in(0,1)$ the monotonicity condition 2) in Theorem \ref{theorem:mainVolisLog2} implies $h''(s)\geq0$ as well if $h''(t)\geq 0$ for all $t\in(1,\infty)$.
\end{proof}

Lemma \ref{lemma:necessary} states that rank-one convexity of an energy with a volumetric-isochoric split always implies the global convexity of $h$ \textbf{or} $f$. Thus if neither $h$ or $f$ are convex, the sum $W(F)=h\bigl(\frac{\lambda_1}{\lambda_2}\bigr)+f(\lambda_1\lambda_2)$ is not rank-one convex. On the other hand, if $h$ and $f$ are both convex, the energy is already polyconvex as the sum of two individual polyconvex energies (cf.\ Lemma \ref{lemma:volumetricConvexity} and Lemma \ref{lemma:convexityCharacterization}). Therefore, for our purpose it is sufficient to consider the two subclasses
\begin{align}
	\Mp&\colonequals\left\{W(F)=h\biggl(\frac{\lambda_1}{\lambda_2}\biggr)+f(\lambda_1\lambda_2)\;|\;h\col(1,\infty)\to\R \text{ is convex}\right\}\\
	\intertext{and}
	\Mm&\colonequals\left\{W(F)=h\biggl(\frac{\lambda_1}{\lambda_2}\biggr)+f(\lambda_1\lambda_2)\;|\;f\col\Rp\to\R \text{ is convex}\right\}
\end{align}
where either $f$ or $h$ are convex. The conditions in Theorem \ref{theorem:mainVolisLog2} contain several different expressions related to the isochoric part $h(t)$, while $f_0=\inf_{z\in(0,\infty)}z^2f''(z)$ occurs exclusively for the volumetric term. We utilize this observation to reduce the volumetric parts of the sets $\Mp$ and $\Mm$ to a single function for all least rank-one convex representatives:
\begin{lemma}\label{lem:volisLog}
	Let $W\col\GLp(2)\to\R$ be a two-times differentiable isotropic energy function with a volumetric-isochoric split
	\[
		W(F)=h\biggl(\frac{\lambda_1}{\lambda_2}\biggr)+f(\lambda_1\lambda_2)
	\]
	that is rank-one convex but not quasiconvex. Then there exists a constant $c\in\R$ so that the energy
	\begin{equation}
		W_0(F)\colonequals h\biggl(\frac{\lambda_1}{\lambda_2}\biggr)+c\.\log(\lambda_1\lambda_2)
	\end{equation}
	is rank-one convex but not quasiconvex, as well.
\end{lemma}
\begin{proof}
	Both energies have the same isochoric part, but different volumetric parts. We set\footnote{We have $\displaystyle f_0=\inf_{z\in(0,\infty)}z^2f''(z)\in\R$ because of the assumed rank-one convexity of $W\col\GLp(2)\to\R$, i.e.\ $f_0+h_0\geq0$.}
	\begin{equation}
		c\colonequals-f_0=-\inf_{z\in(0,\infty)}z^2f''(z)\,,\qquad\ftilde(z)\colonequals c\.\log(z)\,.
	\end{equation}
	Then
	\[
		\ftilde_0=\inf_{z\in(0,\infty)}z^2\ftilde''(z)=\inf_{z\in(0,\infty)}z^2\left(-\frac{c}{z^2}\right)=-c=f_0\,.
	\]
	Thus the conditions in Theorem \ref{theorem:mainVolisLog2} are identical for both energies and they share the same rank-one convexity behavior. Therefore, $W_0(F)$ is rank-one convex. Regarding the quasiconvexity, we consider the difference between the two energies
	\begin{equation}
		W^*(F)\colonequals W(F)-W_0(F)=f(\lambda_1\lambda_2)-c\.\log(\lambda_1\lambda_2)\equalscolon f^*(\lambda_1\lambda_2)\,.\label{eq:W0andWstar}
	\end{equation}
	The energy $W^*(F)$ is a purely volumetric energy with $f_0^*=f_0-\ftilde_0=0$. Thus ${f^*}''(z)\geq 0$ for all $z\in(0,\infty)$, i.e.\ $f^*$ is convex and in particular $W^*(F)$ is quasiconvex (cf.\ Lemma \ref{lemma:volumetricConvexity}). If $W_0(F)$ is quasiconvex, then the sum $W(F)=W_0(F)+W^*(F)$ of two individual quasiconvex functions is quasiconvex as well, which contradicts the initial assumption that $W(F)$ is not quasiconvex.
\end{proof}
Equation \eqref{eq:W0andWstar} shows that every rank-one convex function $W(F)$ with a volumetric-isochoric split can be rewritten as the sum of two individual rank-one convex energies
\begin{equation}
	W(F)=h\biggl(\frac{\lambda_1}{\lambda_2}\biggr)+f(\lambda_1\lambda_2)=h\biggl(\frac{\lambda_1}{\lambda_2}\biggr)+c\.\log(\lambda_1\lambda_2)+f^*(\lambda_1\lambda_2)=W_0(F)+W^*(F)\,.
\end{equation}
The second part $W^*(F)=f^*(\det F)$ is a purely volumetric energy, i.e.\ rank-one convexity already implies quasiconvexity, while the first part $W_0(F)$ maintains the isochoric part $h(t)$. In other words,  Lemma \ref{lem:volisLog} states that \textbf{if} there exists a volumetric-isochorically split energy $W(F)$ that is rank-one convex but not quasiconvex, \textbf{then} there exists a function $W_0(F)=h\bigl(\frac{\lambda_1}{\lambda_2}\bigr)+c\.\log(\lambda_1\lambda_2)$ that is rank-one convex but not quasiconvex as well.

Since multiplying with an arbitrary constant $\alpha\in\Rp$ has no effect on the convexity behavior of the energy function, we can also set $c=\pm1$ without loss of generality. We reduce the two subclasses $\Mp$ and $\Mm$ of volumetric-isochorically split energies to two new sets
\begin{align}
	\Mps&\colonequals\left\{W(F)=h\biggl(\frac{\lambda_1}{\lambda_2}\biggr)+\log(\lambda_1\lambda_2)\;|\;h\col(1,\infty)\to\R \text{ is convex}\right\}\\
	\intertext{and}
	\Mms&\colonequals\left\{W(F)=h\biggl(\frac{\lambda_1}{\lambda_2}\biggr)-\log(\lambda_1\lambda_2)\;|\;h\col(1,\infty)\to\R \text{ is not convex}\right\}.
\end{align}
%
%
%
%
%
\section{Subclass $\Mp$: Convex isochoric part}\label{sec:Mp}
We focus on volumetric-isochorically split energies which have a convex isochoric part $h(t)$ together with a non-convex volumetric part $f(\det F)=\log\det F$, i.e.\ $W\in\Mps$. It is possible to reduce the whole class $\Mps$ and thereby $\Mp$, to a single rank-one convex representative $\Wmp(F)$ (cf.\ Lemma \ref{lem:minimallyConvexEnergy}) as the only candidate to check for quasiconvexity to answer Morrey's questions for all energies with a volumetric-isochoric split and a convex isochoric part.

This candidate is found by searching for energy functions that satisfy inequality conditions of Theorem \ref{theorem:mainVolisLog2} by equality. We start with the first inequality which is the most important condition due to its equivalence to separate convexity: 
\begin{flalign}
	1)&&h_0+f_0\geq 0\,,\qquad h_0=\inf_{t\in(1,\infty)}t^2h''(t)\qquad\text{and}\qquad f_0=\inf_{z\in(0,\infty)}z^2f''(z)\,.&&
\end{flalign}
We consider the case of equality for all $t\geq 1$. For every $W\in\Mps$ the volumetric part is given by $f(z)=\log z$, which implies 
\[
	f_0=\inf_{z\in(0,\infty)}z^2f''(z)=\inf_{z\in(0,\infty)}z^2\.\frac{-1}{z^2}=-1\,,
\]
and we compute
\begin{equation}
		t^2h''(t)-1=0\qquad\iff\qquad h''(t)=\frac{1}{t^2}\qquad\iff\qquad h(t)=-\log t+a\.t+b\,.
\end{equation}
The constant $b\in\R$ is negligible regarding any convexity property and thus we set $b=0$. We determine the other constant $a\in\R$ by testing the remaining convexity conditions and choose the constant $a$ so that these inequalities are satisfied as non-strictly as possible to obtain our first least rank-one convex candidate. In this way, the monotonicity condition 2) of Theorem \ref{theorem:mainVolisLog2} is given by
\begin{equation}
	-\frac{1}{t}+a\geq 0\qquad\text{for all}\quad t\geq 1\qquad\iff\qquad a\geq 1\,.
\end{equation}
Next, condition 3a) of Theorem \ref{theorem:mainVolisLog2},
\begin{align}
	\frac{2\.t}{t-1}\.\frac{a\.t-1}{t}-t^2\.\frac{1}{t^2}-1\geq 0\qquad\iff\qquad\frac{2\.(a-1)\.t}{t-1}\geq 0\qquad\text{for all}\quad t\geq 1\,,
\end{align}
is satisfied for all $a\geq 1$ while condition 4a) of Theorem \ref{theorem:mainVolisLog2},
\begin{align}
	\frac{2\.t}{t+1}\.\frac{a\.t-1}{t}+t^2\.\frac{1}{t^2}+1\geq 0\qquad\iff\qquad\frac{2\.(a+1)\.t}{t+1}\geq 0\qquad\text{for all}\quad t\geq 1\,,
\end{align}
is satisfied for $a\geq -1$. The two remaining conditions 3b) and 4b) of Theorem \ref{theorem:mainVolisLog2} are satisfied for $a\geq 1$ and $a\geq3-2\.\sqrt{2}$, respectively. Putting all conditions together, we conclude that the case $a=1$ should be considered optimal and results in a least rank-one convex candidate. Overall, our canonical candidate is $h(t)=t-\log t$, which corresponds to the energy function

\noindent\fbox{\parbox{0.99\textwidth}{
	\setlength{\abovedisplayskip}{0pt}
	\setlength{\belowdisplayskip}{0pt}
	\begin{equation}
		\Wmp(F)\colonequals\frac{\lambdamax}{\lambdamin}-\log\frac{\lambdamax}{\lambdamin}+\log(\det F)=\frac{\lambdamax}{\lambdamin}+2\.\log\lmin\,.\vspace{-0.15cm}
	\end{equation}
}}

The energy $\Wmp(F)$ satisfies the conditions 1) and 3a) by equality for all $t\geq 1$, conditions 2) and 3b) for $t=1$. We show that this function is the single representative for the energy class $\Mp$ regarding the question of quasiconvexity.

\begin{lemma}\label{lem:minimallyConvexEnergy}
	The following are equivalent:
	\begin{itemize}
	\item[i)]
		There exists a two-times differentiable isotropic energy $W\col\GLp(2)\to\R$ with a volumetric-isochoric split
		\begin{equation}\label{eq:splitEnergy}
			W(F) = h\biggl(\frac{\lambda_1}{\lambda_2}\biggr)+f(\lambda_1\lambda_2)
		\end{equation}
		such that $h$ is convex, i.e.\ $W\in\Mp$, and $W$ is rank-one convex but $W$ is not quasiconvex.
	\item[ii)]
		The rank-one convex function $\Wmp(F)\col\GLp(2)\to\R$ with
		\begin{equation}\label{eq:definitionWzero}
			\Wmp(F)\colonequals \frac{\lambdamax}{\lambdamin} - \log\left(\frac{\lambdamax}{\lambdamin}\right) + \log\det F\vspace{-0.15cm}
		\end{equation}
		with singular values $\lambdamax\geq\lambdamin$ of $F$ is not quasiconvex.
	\end{itemize}
\end{lemma}
\begin{proof}
	Using Theorem \ref{theorem:mainVolisLog2}, we can directly prove the rank-one convexity of $\Wmp$. We calculate
	\[
		h_0=\inf_{t\in(1,\infty)}t^2\.h''(t)=\inf_{t\in(1,\infty)}t^2\frac{1}{t^2}=1\,,\qquad f_0=\inf_{z\in(0,\infty)}z^2\.f''(z)=\inf_{z\in(0,\infty)}z^2\left(-\frac{1}{z^2}\right)=-1\,,
	\]
	\begin{itemize}
		\item[1)] \qquad$\displaystyle h_0+f_0=1-1\geq0\,,$\hfill\checkmark
		\item[2)] \qquad$\displaystyle h'(t)=1-\frac{1}{t}\geq 0$\quad for all $t\geq1\,,$\hfill\checkmark
		\item[3a)] \qquad$\displaystyle \frac{2\.t}{t-1}\.h'(t)-t^2\.h''(t)+f_0=\frac{2\.t}{t-1}\.\frac{t-1}{t}-1-1=0\geq 0\,,$\hfill\checkmark
		\item[4a)] \qquad$\displaystyle \frac{2\.t}{t+1}\.h'(t)+t^2\.h''(t)-f_0=\frac{2\.t}{t+1}\.\frac{t-1}{t}+1+1=2\left(\frac{t-1}{t+1}+1\right)=\frac{4\.t}{t+1}\geq 0$\hfill\checkmark
	\end{itemize}
	 By definition $\Wmp\in\Mp$. i.e.\ ii) $\implies$ i). Starting with $W(F)$ from i), we use Lemma \ref{lem:volisLog} to reduce the class of energies from $\Mp$ to $\Mps$. Thus there exists a constant $c>0$ so that the energy 
	\[
		W_0(F)=h\biggl(\frac{\lambda_1}{\lambda_2}\biggr)+c\.\log(\lambda_1\lambda_2)
	\]
	is rank-one convex but not quasiconvex. Without loss of generality we set $c=1$, because multiplying with an arbitrary positive constant has no effect on the convexity behavior of the energy, and change $h$ accordingly. We consider the difference
	\begin{equation}
		W^*(F)\colonequals W_0(F)-\Wmp(F)= h(t)-t+\log t\equalscolon\htilde(t)\,,\qquad t=\frac{\lmax}{\lmin}\,,
	\end{equation}
	which is a purely isochoric energy. We show that $W^*(F)$ itself is quasiconvex. For this, we study the first two conditions of Theorem \ref{theorem:mainVolisLog2} for the rank-one convex energy
	\[
		W_0(F)=W^*(F)+\Wmp(F)=\htilde(t)+t-\log t+\log\det F\,.
	\]
	We compute $\displaystyle f_0=\inf_{z\in(0,\infty)}z^2\.\frac{-1}{z^2}=-1$ and
	\begin{itemize}
		\item[i)] $\displaystyle h_0+f_0=\inf_{t\in(1,\infty)}t^2\left(\htilde''(t)+\frac{1}{t^2}\right)-1=\htilde_0\geq 0\qquad\implies\qquad\htilde''(t)\geq 0$\quad for all $t\in(1,\infty)$.
		\item[ii)] $h'(t)\geq 0$ for all $t\geq 1$\quad and \quad$\htilde'(t)=h'(t)+1-\frac{1}{t}$\\
		for $t=1$ it holds \quad$\htilde'(1)=h'(1)+1-1=h'(1)\geq 0$.
	\end{itemize}
	Thus $\htilde(t)$ is convex for all $t\geq 1$ and non-decreasing at $t=1$. Both together imply $\htilde'(t)\geq\htilde'(1)\geq 0$ for all $t\geq 1$, i.e.\ the mapping $t\mapsto\htilde(t)$ is non-decreasing on $[1,\infty)$. By Lemma \ref{lemma:convexityCharacterization}, the isochoric energy $W^*(F)=\htilde\bigl(\frac{\lambdamax}{\lambdamin}\bigr)$ is quasiconvex. The sum of two individual quasiconvex functions is again quasiconvex; hence the non-quasiconvexity of $W_0(F)=W^*(F)+\Wmp(F)$ requires that $\Wmp(F)$ is already non-quasiconvex.
\end{proof}
\begin{remark}
	Lemma \ref{lem:minimallyConvexEnergy} can also be stated as the equivalence between the quasiconvexity of $\Wmp$ and the implication
	\[
		W \text{ rank-one convex} \quad\implies\quad W \text{ quasiconvex}\qquad\qquad\text{for all}\quad W\in\Mp\,,
	\]
	i.e.\ for all energies with an additive volumetric-isochoric split and a convex isochoric part.
\end{remark}
Due to the use of the ordered singular values $\lambdamax\geq\lambdamin$ instead of the unsorted singular values $\lambda_1,\lambda_2$ of $F$, the representation $h\col\Rp\to\R$ of the isochoric part is not smooth at $t=1$ in general. Nevertheless, for the candidate $\Wmp$ we have
\begin{equation}
	h(t)=\begin{cases}t-\log t &\caseif t\geq 1\,,\\ \frac{1}{t}+\log t &\caseif t\leq 1\,,\end{cases}
\end{equation}
which implies $h\in C^2(\Rp;\R)$ because
\begin{align}
	\lim_{t\searrow 1}h'(t)&=\lim_{t\to 1}1-\frac{1}{t}=0=\lim_{t\to 1}-\frac{1}{t^2}+\frac{1}{t}=\lim_{t\nearrow 1}h'(t)\,,\notag\\
	\lim_{t\searrow 1}h''(t)&=\lim_{t\to 1}\frac{1}{t^2}=1=\lim_{t\to 1}\frac{2}{t^3}-\frac{1}{t^2}=\lim_{t\nearrow 1}h''(t)\,.\notag
\end{align}
Although Lemma \ref{lem:minimallyConvexEnergy} reduces the question of whether $\Mp$ contains a rank-one convex, non-quasiconvex function to the very particular question of whether $\Wmp(F)$ is quasiconvex, the latter still remains open at this point (cf.\ Section \ref{sec:Burkholder}). We note, however, that the stronger condition of polyconvexity is indeed not satisfied for $\Wmp(F)$
\begin{lemma}\label{lem:WmpNotPolyconvex}
	The energy $\Wmp(F)=\frac{\lambdamax}{\lambdamin} - \log\bigl(\frac{\lambdamax}{\lambdamin}\bigr) + \log\det F$ is not polyconvex.
\end{lemma}
\begin{proof}
	Every polyconvex function $W\col\GLp(2)\to\R$ must grow at least linear for $\det F\to0$ \cite[Cor.~5.9]{Dacorogna08}. This can be motivated by the following: If $W(F)$ is polyconvex, there exists a (not necessarily unique) representation $P\col\R^{2\times2}\times\Rp\to\R$ with $W(F)=P(F,\det F)$ which is convex. We consider the hyperplane attached to $P$ at $(\id,1)$. Convexity in $\R^{2\times2}\times\Rp$ implies that $P$ is bounded below from its hyperplane. On the contrary,
	\[
		\lim_{\lambda\to 0}P(\lambda\.\id,\lambda^2)=\lim_{\lambda\to 0}\Wmp(\lambda\.\id)=\lim_{\lambda\to 0}\frac{\lambda}{\lambda}+2\.\log\lambda=-\infty
	\]
	is unbounded. Thus $\Wmp(F)$ cannot be polyconvex.
\end{proof}
To summarize, the question whether an energy $W\in\Mp$ exists which is rank-one convex but not quasiconvex or not, i.e.\ rank-one convexity always implies quasiconvexity, can be reduced to the specific question of quasiconvexity of the single energy function $\Wmp(F)$. In this sense, our candidate represents an extreme point of $\Mp$, i.e.\ it is the single least rank-one convex function of this subclass. Moreover, our energy $\Wmp(F)$ is non-polyconvex, which is necessary to serve as a possible example for Morrey's conjecture.
%
%
%
%
%
\section{Subclass $\Mm$: Convex volumetric part}\label{sec:Mm}

We continue with volumetric-isochorically split energies which have a non-convex isochoric part $h(t)$ together with a convex volumetric part $f(\det F)=-\log\det F$, i.e.\ $W\in\Mms$. This time, the question of whether an energy $W\in\Mm$ exists which is rank-one convex but not quasiconvex is much more involved and cannot be reduced to the matter of quasiconvexity of a single extreme representative. Nevertheless, we give a few candidates which are analogous to \mbox{$\Wmp\in\Mp$.}

Again, following the concept of least rank-one convex functions, we are looking for energies that cannot be split additively into two non-trivial functions which are both individually rank-one convex. For this, we take the four inequalities of Theorem \ref{theorem:mainVolisLog2} characterizing rank-one convexity and consider the corresponding differential equations, i.e.\ energies which only fulfill these inequalities conditions by equality. However, a function does not have to satisfy all four conditions non-strictly to be a least rank-one convex candidate; recall that the energy $\Wmp(F)$ as the single representative of $\Mp$ satisfies conditions 1) and 3a) by equality for all $t\geq 1$ and condition 2) for $t=1$ but is not optimal for the remaining inequalities.

As a method to find candidates for least rank-one convex functions of $\Mm$, we examine each condition characterizing rank-one convexity separately and search for equality for all $t\geq 1$. Thus we are capable of finding very interesting energies, e.g.\ a function which is nowhere strictly separately convex or satisfies another inequality condition by equality. 
However, due to their dependence on $t\geq 1$ this approach does not cover all possible extreme cases of $\Mm$. An energy can already be a least rank-one convex candidate if it satisfies one or more inequality conditions non-strictly for a single $t\geq 1$ and not all $t\geq 1$ as described above (cf.\ Section \ref{sec:Ws}). Besides, conditions 3) and 4) allow for even more involved combinations because of their internal case-by-case distinction.
%
%
%
\subsection{Obtaining $\Wmm$}\label{sec:Wmm}
We start with the first inequality of Theorem \ref{theorem:mainVolisLog2}, which is the most important condition due to its equivalence to separate convexity: 
\begin{flalign}
	1)&&h_0+f_0\geq 0\,,\qquad h_0=\inf_{t\in(1,\infty)}t^2h''(t)\qquad\text{and}\qquad f_0=\inf_{z\in(0,\infty)}z^2f''(z)\,.&&
\end{flalign}
We consider the case of equality for all $t\geq 1$. For every $W\in\Mms$ the volumetric part is given by $f(z)=-\log z$, which implies 
\[
	f_0=\inf_{z\in(0,\infty)}z^2f''(z)=\inf_{z\in(0,\infty)}z^2\.\frac{1}{z^2}=1\,,
\]
and we compute
\begin{equation}
		t^2h''(t)+1=0\qquad\iff\qquad h''(t)=-\frac{1}{t^2}\qquad\iff\qquad h(t)=\log t+a\.t+b\,.
\end{equation}
The constant $b\in\R$ is negligible regarding any convexity property and thus we set $b=0$. We determine the other constant $a\in\R$ by testing the remaining convexity conditions and choose the constant $a$ so that these inequalities are satisfied as non-strictly as possible to obtain our first least rank-one convex candidate. In this way, the monotonicity condition 2) of Theorem \ref{theorem:mainVolisLog2} is given by
\begin{equation}
	\frac{1}{t}+a\geq 0\qquad\text{for all}\quad t\geq 1\qquad\iff\qquad a\geq 0\,.
\end{equation}
Next, condition 3a) of Theorem \ref{theorem:mainVolisLog2},
\begin{align}
	\frac{2\.t}{t-1}\.\frac{1+a\.t}{t}+t^2\.\frac{1}{t^2}+1\geq 0\qquad\iff\qquad\frac{2\.(a+1)\.t}{t-1}\geq 0\qquad\text{for all}\quad t\geq 1\,,
\end{align}
is satisfied for all $a\geq -1$ while condition 4a) of Theorem \ref{theorem:mainVolisLog2},
\begin{align}
	\frac{2\.t}{t+1}\.\frac{1+a\.t}{t}-t^2\.\frac{1}{t^2}-1\geq 0\qquad\iff\qquad\frac{2\.(a-1)\.t}{t+1}\geq 0\qquad\text{for all}\quad t\geq 1\,,
\end{align}
is satisfied for $a\geq 1$. The two remaining conditions 3b) and 4b) of Theorem \ref{theorem:mainVolisLog2} are never satisfied for $a\neq 1$. Putting all conditions together, we conclude that the case $a=1$ yields a least rank-one convex function. Overall, our canonical candidate regarding equality in condition 1) of Theorem \ref{theorem:mainVolisLog2} for all $t\geq1$ is $h(t)=t+\log t$, which corresponds to the energy function

\fbox{\parbox{0.99\textwidth}{
	\setlength{\abovedisplayskip}{0pt}
	\setlength{\belowdisplayskip}{0pt}
	\begin{equation}
		\Wmm(F)\colonequals\frac{\lambdamax}{\lambdamin}+\log\frac{\lambdamax}{\lambdamin}-\log(\det F)\,,\vspace{-0.15cm}
	\end{equation}
}}

an analogous candidate to $\Wmp\in\Mp$. The energy $\Wmm(F)$ satisfies condition 1), 4a) and 4b) non-strict for all $t\geq 1$, conditions 2) in the limit $t\to\infty$ but is not optimal for the remaining inequalities. Contrary to the discussion in the last section, where we showed that $\Wmp$ is as the single least rank-one convex candidate for the energy class $\Mp$, the same is not true for $\Wmm(F)$, i.e.\ it exists several distinct least rank-one convex energies for the class $\Mm$. 
%
%
%
\subsection{Obtaining $\Ws$}\label{sec:Ws}
We continue with the second condition of Theorem \ref{theorem:mainVolisLog2}, which is equivalent to the Baker-Ericksen inequalities, and consider the limit case
\begin{flalign}
	2)&&h'(t)=0\qquad\text{for all}\quad t\geq 1\,.&&
\end{flalign}
The only solution to this differential equation is the constant function $h(t)=c$ which provides the purely volumetric energy $W(F)=c-\log\det F$. Thus we have no (non-trivial) least rank-one convex candidate corresponding to the monotonicity condition 2). 

Before we continue to work with the more involved conditions 3) and 4), we will consider an alternative variant to produce more least rank-one convex energies utilizing the fact that an energy function can be least rank-one convex without satisfying one single condition non-strict for all $t\in(1,\infty)$. Starting with $\Wmm(F)$, we can weaken the monotonicity condition 2) by waiving the requirement of being nowhere strictly separately convex. However, we have to ensure that the energy remains non-strictly separately convex for one $t\in[1,\infty]$. Thus we are looking for a new function $h(t)$ which is less monotone but slightly more convex then the expression $t+\log t$. We choose the ansatz
\begin{equation}
	W(F)=a\.t+\frac{b}{t}+\log t\,,\qquad a,b\in\R\label{eq:ansatzWs}
\end{equation}
and with the first two conditions of Theorem \ref{theorem:mainVolisLog2} we compute
\begin{flalign*}
	1)&&h_0+f_0&=\inf_{t\in(0,\infty)}t^2\left(\frac{2\.b}{t^3}-\frac{1}{t^2}\right)+1=\inf_{t\in(0,\infty)}\frac{4\.b}{t}\geq 0\qquad\iff\qquad b\geq0\,,&&\\
	2)&&h'(t)&=a-\frac{b}{t^2}+\frac{1}{t}\geq 0\qquad\iff\qquad a\geq0\quad\text{and}\quad b\leq a+1\,.&&
\end{flalign*}
We are interested in the limit case $b=a+1$ so that $h'(1)=0$, i.e.\ the function is not strictly monotone for $t=1$. The second limit case $a=0$ would reestablish $\Wmm(F)$. Note that the separate convexity condition is only non-strict in the limit case $t\to\infty$. We evaluate the remaining conditions and choose the remaining constant $a\geq 0$ accordingly. Then
\begin{flalign*}
	3a)&&\frac{2\.t}{t-1}\left(a-\frac{a+1}{t^2}+\frac{1}{t}\right)-t^2\left(2\.\frac{a+1}{t^3}-\frac{1}{t^2}\right)+1\geq 0\quad\iff\quad 2\.(a+1)\geq 0\quad\text{for all}\quad t\geq 1\,,&&
\end{flalign*}
which is satisfied for $a\geq -1$ which is already ensured by conditions 2). Lastly, we compute
\begin{flalign*}
	4a)&&\frac{2\.t}{t+1}\left(a-\frac{a+1}{t^2}+\frac{1}{t}\right)+t^2\left(2\.\frac{a+1}{t^3}-\frac{1}{t^2}\right)-1\geq 0\quad\iff\quad 2\.\frac{a\.t-t+a+1}{t+1}\geq 0\quad\text{for all}\quad t\geq 1&&
\end{flalign*}
which is fulfilled for $a\geq 1$. For the case $0\leq a\leq 1$ the remaining condition
\begin{flalign}
	4b)&&2\.\frac{1+t+3\.t^2-t^3+a\.(1+t+3\.t^2+3\.t^3)}{t^2}\geq 0\qquad\text{for all}\quad t\geq 1&&
\end{flalign}
is satisfied for $a\geq\frac{1}{3}$. Overall, $a=\frac{1}{3}$ yields the least rank-one convex candidate with the isochoric term $h(t)=\frac{t}{3}+\frac{4}{3\.t}+\log t$ which generates

\fbox{\parbox{0.99\textwidth}{
	\setlength{\abovedisplayskip}{0pt}
	\setlength{\belowdisplayskip}{0pt}
	\begin{equation}
		\Ws(F)\colonequals\frac{1}{3}\.\frac{\lambdamax}{\lambdamin}+\frac{4}{3}\frac{\lambdamin}{\lambdamax}+\log\frac{\lambdamax}{\lambdamin}-\log(\det F)\,.\vspace{-0.15cm}
	\end{equation}
}}

The energy $\Ws(F)$ is strictly rank one convex for all $t\in(1,\infty)$, but unlike $\Wmm(F)$, it is only a least rank-one candidate in the sense that it is non-strictly separately convex for $t\to\infty$ and satisfies condition 2) of Theorem \ref{theorem:mainVolisLog2} by equality for $t=1$. 
%
%
%
\subsection{Obtaining $\Wn$}\label{sec:Wn}
We generate a third candidate using our original method with focus on the last two inequalities of Theorem \ref{theorem:mainVolisLog2} and start with condition
\begin{flalign}
	3a)&&\frac{2t}{t-1}\.h'(t)-t^2\.h''(t)+1=0\qquad\iff\qquad h''(t)=\frac{2}{t\.(t-1)}\.h'(t)+\frac{1}{t^2}\,.&&
\end{flalign}
We solve the differential equation by adding a particular solution to the associated homogeneous equation:
\begin{align}
	&&h'(t)&=x_h(t)+x_p(t)\qquad\text{with}\qquad x_p(t)=\frac{1-t}{t^2}\qquad\text{and}\\
	&&x_h'(t)&=2\left(\frac{1}{t-1}-\frac{1}{t}\right)x_h(t)\notag\\
	&\iff&\log x_h(t)&=2\left(\log|t-1|-\log t\right)+c\notag\\
	&\iff& x_h(t)&=c\left(\frac{t-1}{t}\right)^2\,.
\end{align}
Thus
\begin{equation}
	h'(t)=c\left(1-\frac{2}{t}+\frac{1}{t^2}\right)-\frac{1}{t}+\frac{1}{t^2}\qquad\iff\qquad h(t)=c\left(t-2\.\log t-\frac{1}{t}\right)-\log t-\frac{1}{t}+b\,.
\end{equation}
However, the monotonicity condition
\begin{equation}
	h'(t)=\frac{c\.(t-1)^2-(t-1)}{t^2}\geq 0\qquad\iff\qquad c\geq\frac{1}{t-1}\qquad\text{for all}\quad t>1
\end{equation}
cannot be satisfied in the limit $t\to1$. Thus there does not exist a rank-one convex candidate which satisfies condition 3a) by equality for all $t>1$. We continue with condition
\begin{flalign}
	4a)&&\qquad\frac{2t}{t+1}\.h'(t)+t^2\.h''(t)-1=0\qquad\iff\qquad h''(t)=-\frac{2}{t\.(t-1)}\.h'(t)+\frac{1}{t^2}\,.&&
\end{flalign}
Again, we solve the differential equation by adding a particular solution to the associated homogeneous equation:
\begin{align}
	&&h'(t)&=x_h(t)+x_p(t)\qquad\text{with}\qquad x_p(t)=-\frac{t+1}{t^2}\qquad\text{and}\\
	&&x_h'(t)&=2\left(\frac{1}{t+1}-\frac{1}{t}\right)x_h(t)\notag\\
	&\iff&\log x_h(t)&=2\left(\log(t+1)-\log t\right)+c\notag\\
	&\iff& x_h(t)&=c\left(\frac{t+1}{t}\right)^2.
\end{align}
Thus
\begin{equation}
	h'(t)=c\left(1+\frac{2}{t}+\frac{1}{t^2}\right)-\frac{1}{t}-\frac{1}{t^2}\qquad\iff\qquad h(t)=c\left(t+2\.\log t-\frac{1}{t}\right)-\log t+\frac{1}{t}+b\,.
\end{equation}
We set $b=0$ and determine $c\in\R$ by evaluating the remaining rank-one convexity conditions. The separate convexity condition 1) yields
\begin{align}
	h_0&=\inf_{t\in(1,\infty)}t^2h''(t)=t^2\left[c\left(-\frac{2}{t^3}-\frac{2}{t^2}\right)+\frac{2}{t^3}+\frac{1}{t^2}\right]=\inf_{t\in(1,\infty)}2\.(1-c)\.\frac{1}{t}+1-2\.c \geq-1\,.
\end{align}
In the case $c\geq 1$, the infimum is attained at $t=1$ and
\[
	3-4\.c\geq-1\qquad\iff\qquad c\leq 1\,.
\]
For the case $c\leq 1$, the infimum is at $t\to\infty$ and
\[
	1-2\.c\geq-1\qquad\iff\qquad c\leq 1\,.
\]
Thus $c\leq 1$. The monotonicity condition
\begin{flalign}
	2)&&h'(t)=\frac{c\.(t+1)^2-(t+1)}{t^2}\geq 0\qquad\iff\qquad c\geq\frac{1}{t+1}\qquad\text{for all}\quad t\geq 1&&
\end{flalign}
is satified for $c\geq\frac{1}{2}$. The remaining condition 3a) leads to
\begin{align}
	&&\frac{2\.t}{t-1}\.\frac{c\.(t+1)^2-(t+1)}{t^2}-\left[c\left(-\frac{2}{t^3}-\frac{2}{t^2}\right)+\frac{2}{t^3}+\frac{1}{t^2}\right]+1&\geq 0\notag\\
	&\iff& 2\.\frac{c\.(t+1)^2-(t+1)}{t\.(t-1)}+2\.(c-1)\.\frac{t+1}{t}&\geq 0\notag\\
	&\iff& \frac{c\.(t+1)^2-(t+1)+(c-1)\.(t+1)(t-1)}{t-1}&\geq 0\\
	&\iff& \frac{c\.(2\.t^2+2\.t)-t^2-t}{t-1}&\geq 0\notag\\
	&\iff& 2\.c-1&\geq 0\qquad\iff\qquad c\geq\frac{1}{2}\notag\,.
\end{align}
Thus we can consider the two borderline cases
\begin{equation}
	c=\frac{1}{2}\col\qquad h(t)=\frac{1}{2}\left(t+\frac{1}{t}\right),\qquad\text{and}\qquad c=1\col\qquad h(t)=t+\log t\,.
\end{equation}
However, the first function
\begin{equation}
	h\left(\frac{\lambda_1}{\lambda_2}\right)=\frac{1}{2}\left(\frac{\lambda_1}{\lambda_2}+\frac{\lambda_2}{\lambda_1}\right)=\frac{\lambda_1^1+\lambda_2^2}{2\.\lambda_1\lambda_2}=\frac{\norm{F}^2}{2\det F}=\K(F)
\end{equation}
is already polyconvex and therefore not suitable to provide an example for Morrey's question with respect to volumetric-isochorically split energies. The second part $h(t)=t+\log t$ is already known as the iscchoric part of $\Wmm(F)$.

We consider the condition
\begin{flalign}
	3b)&& t^2(t^2-1)\.h'(t)h''(t)-2t\.h'(t)^2+(t^2-4t+3)\.h'(t)+2t\.(t-1)^2\.h''(t)=0\,.&&
\end{flalign}
With the support of Mathematica, the only solution that also satisfies the remaining inequalities turns out to be
\begin{align}
	h(t)&=\frac{1}{2c\.t}\left[1+\sqrt{4c\.t+(1+t)^2}+t^2-2\.(1+2\.\sqrt{1+c})\.t+\sqrt{1+2\.(1+2c)\.t+t^2}\right]+2\.\log\left(2\.(1+c+\sqrt{1+c})\right)\notag\\
	&\phantom{=}\;+\log t-\log\left(1-2c+t+\sqrt{4c\.t+(1+t)^2}\right)-\log\left(1+(1+2c)\.t+\sqrt{4c\.t+(1+t)^2}\right)
\end{align}
with an arbitrary constant $c>0$. The condition 1) is non-strict in the limit $t\to\infty$ and condition 2) is non-strict for $t=1$ which is a similar behavior to that of $\Ws(F)$. In the case of $c=1$, condition 4a) is non-strict in the limit $t\to\infty$ and we obtain
\begin{equation}
	h(t)=\frac{1}{2\.t}\left(1+t^2+(1+t)\.\sqrt{1+6\.t+t^2}\right)-\log t+\log\left(1+4\.t+t^2+(1+t)\.\sqrt{1+6\.t+t^2}\right)
\end{equation}
as our new least rank-one convex candidate. Note that the relation
\begin{align}
	h(t)=\frac{1}{2}\left(t+\frac{1}{t}+\sqrt{2+t^2+\frac{1}{t^2}+6\left(t+\frac{1}{t}\right)}\right)+\log\left(4+t+\frac{1}{t}+\sqrt{2+t^2+\frac{1}{t^2}+6\left(t+\frac{1}{t}\right)}\right)=h\left(\frac{1}{t}\right)\notag
\end{align}
implies that $h$ is smooth at $t=1$. Together with the volumetric part $f(z)=-\log z$ we arrive at

\fbox{\parbox{0.99\textwidth}{
	\setlength{\abovedisplayskip}{0pt}
	\setlength{\belowdisplayskip}{0pt}
	\begin{align}
		\Wn(F)&\colonequals\frac{1}{2}\left(\frac{\lmax}{\lmin}+\frac{\lmin}{\lmax}+\mathcal{F}\left(\frac{\lmax}{\lmin}\right)\right)+\log\left(4+\frac{\lmax}{\lmin}+\frac{\lmin}{\lmax}+\mathcal{F}\left(\frac{\lmax}{\lmin}\right)\right)-\log\det F\,,\notag\\
		\mathcal{F}(t)&=\sqrt{2+t^2+\frac{1}{t^2}+6\left(t+\frac{1}{t}\right)},\vspace{-0.15cm}
	\end{align}
}}

which is equivalent to
\begin{align*}
	\Wn(F)=\K(F)+\sqrt{\K^2(F)+3\.\K(F)}+\log\left(2+\K(F)+\sqrt{\K^2(F)+3\.\K(F)}\right)+\log 2-\log\det F\,.
\end{align*}
For the last differential equation of Theorem \ref{theorem:mainVolisLog2},
\begin{flalign}
	4b)&& t^2(t^2-1)\.h'(t)h''(t)-2t\.h'(t)^2+(t^2+4t+3)\.h'(t)+2t\.(t+1)^2\.h''(t)=0\,,&&
\end{flalign}
we use Mathematica again, but this time the already known candidate $\Wmm(F)$ is the only solution which also satisfies the remaining conditions as a least rank-one convex candidate. 
%
%
%
\subsection{Comparison}
Overall, we have obtained three energy function $\Wmm(F), \Ws(F)$ and $\Wn(F)$ which satisfy different conditions of Theorem \ref{theorem:mainVolisLog2} by equality (cf.\ Table \ref{fig:Mm_candidates} and Figure \ref{fig:Mm_candidates2}). We use them as least rank-one convex representatives of $\Mm$.
\begin{table}[h!]
	\centering
	\renewcommand{\arraystretch}{1.5}
	\begin{tabular}{l|c|c|cc|cc}
		Energy & 1) & 2) & 3a) & 3b) & 4a) & 4b)\\\hline
 		$\Wmm(F)$ & {\color{ForestGreen}$\equiv$} & $\checkmark$ & $\checkmark$ & \xmark & {\color{ForestGreen}$\equiv$} & {\color{ForestGreen}$\equiv$}\\
 		$\Ws(F)$ & {\color{ForestGreen}$t\to\infty$} & {\color{ForestGreen}$t=1$} & $\checkmark$ & \xmark & \xmark & $\checkmark$\\
 		$\Wn(F)$ & {\color{ForestGreen}$t\to\infty$} & {\color{ForestGreen}$t=1$} & $\checkmark$ & {\color{ForestGreen}$\equiv$} & {\color{ForestGreen}$t\to\infty$} & $\checkmark$\\\hline
 		$\Wmp(F)$ & {\color{ForestGreen}$\equiv$} & {\color{ForestGreen}$t=1$} & {\color{ForestGreen}$\equiv$} & {\color{ForestGreen}$t=1$} & $\checkmark$ & $\checkmark$
	\end{tabular}
	\caption{All three candidates $\Wmm(F),\Ws(F)$ and $\Wn(F)\in\Mm$ satisfy different conditions of Theorem \ref{theorem:mainVolisLog2} non-strictly; the {\color{ForestGreen}$\equiv$} indicates equality for all $t\geq 1$. Here, $\checkmark$ represents that the inequality condition is strictly satisfied (and not by equality) while \xmark\ states that condition is not satifsied for all $t\in(0,\infty)$. Only one of the two conditions 3a) and 3b) and one of the two conditions 4a) and 4b) have to be fulfilled.}\label{fig:Mm_candidates}
\end{table}
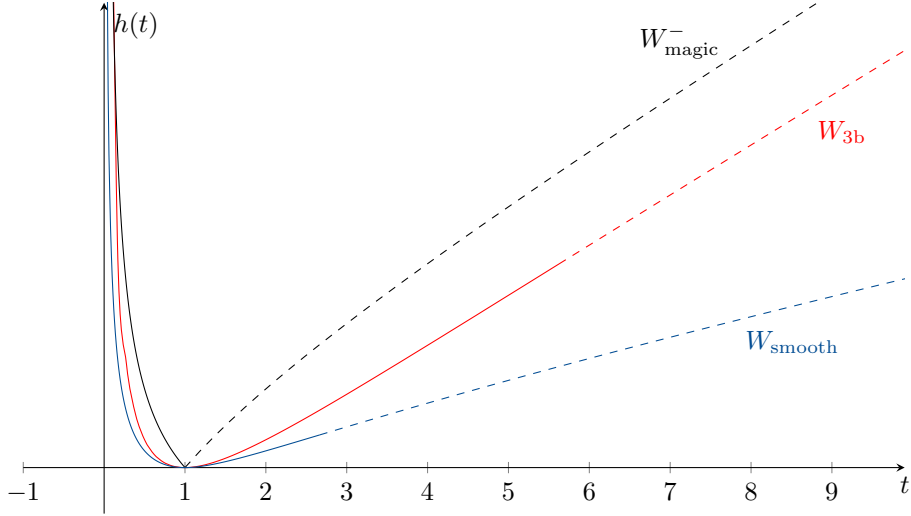
\begin{figure}[h!]
  	\centering
    \begin{tikzpicture}
		\begin{axis}[
        axis x line=middle,axis y line=middle,
        x label style={at={(current axis.right of origin)},anchor=north, below},
        xlabel=$t$, ylabel=$h(t)$,
        xmin=-1, xmax=9.9,
        ymin=-0.99, ymax=10,
        width=\singleGraph\linewidth,
        height=\singleGraph*\graphRatio\linewidth,
        ytick=\empty,
        ]
        \addplot[red, smooth][domain=0.05:5.6262, samples=\sample]{1/(2*x)*(1+x^2+(1+x)*sqrt(1+6*x+x^2))-ln(x)+ln(1+4*x+x^2+(1+x)*sqrt(1+6*x+x^2))-6.28432};
        \addplot[red, dashed][domain=5.6262:10, samples=\sample]{1/(2*x)*(1+x^2+(1+x)*sqrt(1+6*x+x^2))-ln(x)+ln(1+4*x+x^2+(1+x)*sqrt(1+6*x+x^2))-6.28432} node[below,pos=0.8,yshift = -0.3cm] {$\Wn$};
        \addplot[black, smooth][domain=0.05:1, samples=\sample]{1/x-ln(x)-1};
        \addplot[black, dashed][domain=1:10, samples=\sample]{x+ln(x)-1} node[above,pos=0.75,xshift = -0.5cm] {$\Wmm$};
       	\addplot[udedarkblue, smooth][domain=0.03:1, samples=\sample]{1/(3*x)+4*x/3-ln(x)-5/3};
        \addplot[udedarkblue, smooth][domain=1:8/3, samples=\sample]{x/3+4/(3*x)+ln(x)-5/3};
        \addplot[udedarkblue, dashed][domain=8/3:10, samples=\sample]{x/3+4/(3*x)+ln(x)-5/3} node[below,pos=0.8,yshift = -0.2cm] {$\Ws$};
        \end{axis}
  \end{tikzpicture}
  \caption{Visualization of the isochoric part $h\col\Rp\to\R$ with $h(t)=h\bigl(\frac1t\bigr)$ of the three candidates $\Wmm, \Ws$ and $\Wn\in\Mm$. The non-convex regions of $h$ are shown by the dashed lines.}\label{fig:Mm_candidates2}
\end{figure}

Unfortunately, we cannot show yet that we identified all least rank-one convex energies in the class $\Mm$. Similar to Section \ref{sec:Ws} it is possible to construct more least rank-one convex energies, in the sense that they satisfy at least one inequality condition of Theorem \ref{theorem:mainVolisLog2} non-strictly for an arbitrary $t\in[1,\infty)$ or in the limit $t\to\infty$ by choosing a different ansatz than equation \eqref{eq:ansatzWs}. Therefore, it is not sufficient to test these three candidates for quasiconvexity and conclude the quasiconvexity behavior for the whole energy class $\Mm$. This has the following reasons: First, we only considered the case where the isochoric part satisfies one of the conditions 3a) and 4a) for all $t\geq 1$ or 3b) and 4b) for all $t\geq 1$. This is a more restrictive constraint than needed for rank-one convexity, e.g.\ it would be sufficient if an energy satisfies condition 3a) for some interval $I\subset\Rp$ and condition 3b) for the remaining part $\Rp\setminus I$. It is only necessary that for every $t\geq 1$ one of the two conditions 3a) or 3b) (and 4a) or 4b) respectively) is satisfied.

It is not clear how to decide which of our three energies $\Wmm,\,\Ws,\,\Wn$ is the best candidate for Morrey's conjecture because each energy is already least rank-one convex in the sense of satisfying one inequality condition of Theorem \ref{theorem:mainVolisLog2} non-strictly. Thus the difference between two least rank-one convex candidates is neither rank-one convex nor a rank-one concave function, e.g.\ the transition from $\Wmm$ to $\Ws$ is done by adding the isochoric part $h(t)=\frac{4}{3}\.\frac{1}{t}-\frac{2}{3}\.t$ which is convex but monotone decreasing. An energy function can also be non-rank-one convex and non-concave simultaneously at some $F_0\in\GLp(2)$, similar to a saddle point. This makes it very challenging, if not impossible, to find a single weakest rank-one convex candidate for $\Mm$.

Still, with separate convexity as the most important part of rank-one convexity, it is reasonable to start with the candidate $\Wmm(F)$ and check for quasiconvexity. Our first (negative) result in this respect is the following.
\begin{lemma}\label{lem:WmmPolyconvex}
	The energy
	\[
		\Wmm\col\GLp(2)\to\R\,,\quad \Wmm(F) = \frac{\lambdamax}{\lambdamin} + \log\left( \frac{\lambdamax}{\lambdamin} \right) - \log(\lambdamax\.\lambdamin) = \frac{\lambdamax}{\lambdamin} - 2\.\log(\lambdamin)\,.
	\]
	is polyconvex.
\end{lemma}
\begin{proof}	
	We check for polyconvexity\index{Polyconvex!function} with a Theorem due to \v{S}ilhav\'{y} \cite[Proposition 4.1]{vsilhavy2002convexity} (see also \cite[Theorem 4.1]{mielke2005necessary}) for energy functions in terms of the (ordered) singular values. It states that any $W\col\GLp(2)\to\R$ of the form
	\begin{equation}\label{eq:mielkeCriterionRepresentation}
		W(F) = \ghat(\lambdamax,\lambdamin) \qquad\text{for all $F\in\GLp(2)$ with singular values }\lambdamax\geq\lambdamin
	\end{equation}
	is polyconvex if and only if for every $\gamma_1\geq\gamma_2>0$, there exists
	\begin{equation}\label{eq:mielkeCriterionInterval}
		c\in \left[ -\frac{f_1-f_2}{\gamma_1-\gamma_2}\,,\; \frac{f_1+f_2}{\gamma_1+\gamma_2} \right],\qquad f_i = \dd{\gamma_i}\.\ghat(\gamma_1,\gamma_2)\,,
	\end{equation}
	such that for all $\nu_1\geq\nu_2>0$,
	\begin{equation}\label{eq:mielkeCriterionInequality}
		\ghat(\nu_1,\nu_2) \geq \ghat(\gamma_1,\gamma_2) + f_1\.(\nu_1-\gamma_1) + f_2\.(\nu_2-\gamma_2) + c\.(\nu_1-\gamma_1)\.(\nu_2-\gamma_2)\,.
	\end{equation}
	The energy $\Wmm$ can be expressed in the form \eqref{eq:mielkeCriterionRepresentation} with $\ghat$ such that
	\[
		\ghat(x,y) = \frac{x}{y} - 2\.\log(y)
	\]
	for $x>y>0$. In particular, for $\gamma_1\geq\gamma_2>0$,
	\begin{align}
		f_1 &= \dd{\gamma_1}\.\ghat(\gamma_1,\gamma_2) = \dd{\gamma_1}\left[\frac{\gamma_1}{\gamma_2} - 2\.\log(\gamma_2) \right] = \frac{1}{\gamma_2}\,,\notag\\
		f_2 &= \dd{\gamma_2}\.\ghat(\gamma_1,\gamma_2) = \dd{\gamma_2}\left[\frac{\gamma_1}{\gamma_2} - 2\.\log(\gamma_2)\right] = -\frac{2}{\gamma_2} - \frac{\gamma_1}{\gamma_2^2} = -\frac{2\.\gamma_2+\gamma_1}{\gamma_2^2}\,.
	\end{align}
	Now, for any given $\gamma_1\geq\gamma_2>0$, choose
	\[
		c = \frac{f_1+f_2}{\gamma_1+\gamma_2} = \frac{\frac{1}{\gamma_2}-\frac{2\gamma_2+\gamma_1}{\gamma_2^2}}{\gamma_1+\gamma_2} = \frac{\frac{-\gamma_2-\gamma_1}{\gamma_2^2}}{\gamma_1+\gamma_2} = - \frac{1}{\gamma_2^2}\cdot\frac{\gamma_1+\gamma_2}{\gamma_1+\gamma_2} = -\frac{1}{\gamma_2^2}\,.
	\]
	In order to establish the proposition, it suffices to show that for any $\nu_1\geq\nu_2>0$,
	\begin{align}
		0&\leq\ghat(\nu_1,\nu_2) - \ghat(\gamma_1,\gamma_2) - f_1\.(\nu_1-\gamma_1) - f_2\.(\nu_2-\gamma_2) - c\.(\nu_1-\gamma_1)\.(\nu_2-\gamma_2)\notag\\
		&= \frac{\nu_1}{\nu_2} - 2\.\log(\nu_2) - \frac{\gamma_1}{\gamma_2} + 2\.\log(\gamma_2) - f_1\.(\nu_1-\gamma_1) - f_2\.(\nu_2-\gamma_2) - c\.(\nu_1-\gamma_1)\.(\nu_2-\gamma_2)\notag\\
		&= \frac{\nu_1}{\nu_2} - \frac{\gamma_1}{\gamma_2} + 2\.\log\left(\frac{\gamma_2}{\nu_2}\right) - \frac{1}{\gamma_2}\cdot(\nu_1-\gamma_1) + \frac{2\.\gamma_2+\gamma_1}{\gamma_2^2}\cdot(\nu_2-\gamma_2) + \frac{1}{\gamma_2^2}\cdot (\nu_1-\gamma_1)\.(\nu_2-\gamma_2)\notag\\
		&= \frac{\nu_1}{\nu_2} - \frac{\gamma_1}{\gamma_2} + 2\.\log\left(\frac{\gamma_2}{\nu_2}\right) - \frac{\nu_1}{\gamma_2} + \frac{\gamma_1}{\gamma_2} + \frac{1}{\gamma_2^2}\cdot \bigl( (2\.\gamma_2+\gamma_1)\.(\nu_2-\gamma_2) + (\nu_1-\gamma_1)\.(\nu_2-\gamma_2) \bigr)\\
		&= \frac{\nu_1}{\nu_2} - \frac{\nu_1}{\gamma_2} + 2\.\log\left(\frac{\gamma_2}{\nu_2}\right) + \frac{1}{\gamma_2^2}\.(2\.\gamma_2+\nu_1)\.(\nu_2-\gamma_2)\notag\\
		&= \nu_1\cdot \left( \frac{1}{\nu_2} - \frac{1}{\gamma_2}  + \frac{\nu_2-\gamma_2}{\gamma_2^2} \right) + 2\.\log\left(\frac{\gamma_2}{\nu_2}\right) + \frac{1}{\gamma_2^2}\cdot2\.\gamma_2\cdot(\nu_2-\gamma_2)\notag\\
		&= \nu_1\cdot\frac{\gamma_2^2-\gamma_2\.\nu_2+\nu_2^2-\gamma_2\.\nu_2}{\gamma_2^2\.\nu_2} + 2\.\log\left(\frac{\gamma_2}{\nu_2}\right) + \frac{2\.(\nu_2-\gamma_2)}{\gamma_2}\notag\\
		&= \nu_1\cdot\frac{(\gamma_2-\nu_2)^2}{\gamma_2^2\.\nu_2} + 2\.\log\left(\frac{\gamma_2}{\nu_2}\right) + 2\.\frac{\nu_2}{\gamma_2} - 2\,.\notag
	\end{align}
	Since $\nu_1\geq\nu_2>0$ by assumption,
	\begin{align}
		&\hspace*{-2.1em}\nu_1\cdot\frac{(\gamma_2-\nu_2)^2}{\gamma_2^2\.\nu_2} + 2\.\log\left(\frac{\gamma_2}{\nu_2}\right) + 2\.\frac{\nu_2}{\gamma_2} - 2\notag\\
		&\geq \nu_2\cdot\frac{(\gamma_2-\nu_2)^2}{\gamma_2^2\.\nu_2} + 2\.\log\left(\frac{\gamma_2}{\nu_2}\right) + 2\.\frac{\nu_2}{\gamma_2} - 2\notag\\
		&= \frac{\gamma_2^2-2\.\gamma_2\.\nu_2+\nu_2^2}{\gamma_2^2} + 2\.\log\left(\frac{\gamma_2}{\nu_2}\right) + 2\.\frac{\nu_2}{\gamma_2} - 2\\
		&= 1 + \left( \frac{\nu_2}{\gamma_2} \right)^2 - 2\.\frac{\nu_2}{\gamma_2} + 2\.\log\left(\frac{\gamma_2}{\nu_2}\right) + 2\.\frac{\nu_2}{\gamma_2} - 2\notag\\
		&= \left( \frac{\nu_2}{\gamma_2} \right)^2 + 2\.\log\left(\frac{\gamma_2}{\nu_2}\right) - 1
		\;=\; \left( \frac{\nu_2}{\gamma_2} \right)^2 - \log\left(\left(\frac{\nu_2}{\gamma_2}\right)^2\right) - 1 \;\geq\; 0\,,\notag
	\end{align}
	where the final inequality holds since $t\geq1+\log(t)$ for all $t>0$.
\end{proof}
\newgeometry{top=15mm, left=20mm}
\newpage
\begin{figure}[h!]
	\includegraphics[clip, trim=2.5cm 13cm 6cm 13cm, width=1.1\textwidth]{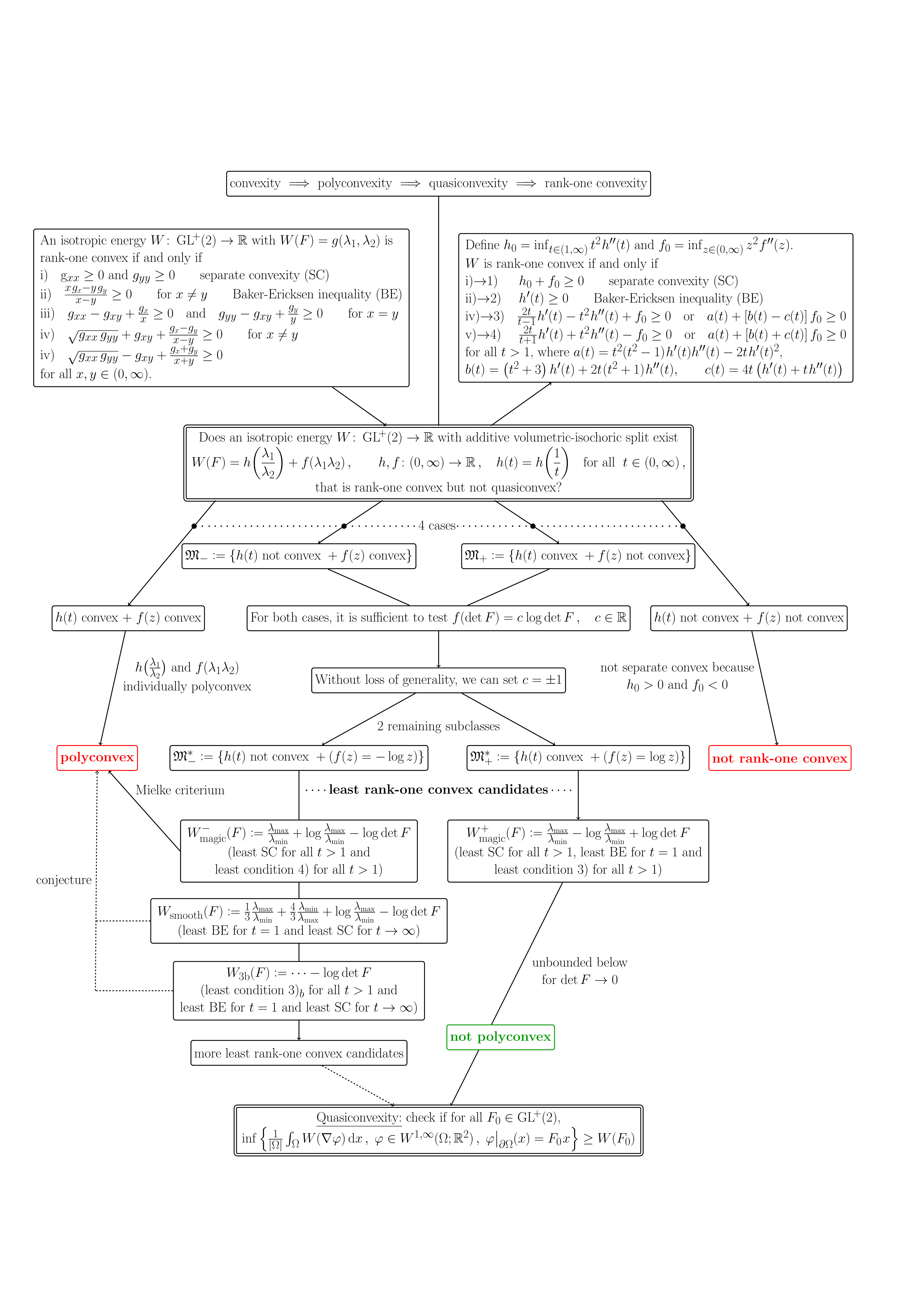}
	\caption{\normalsize Overview of the reduction from Morrey's question for the class of isotropic planar energies with an additive volumetric-isochoric split to the question whether or not the least rank-one convex candidates are quasiconvex. For energies with a convex isochoric term $h(t)$ it is sufficient to test the non-polyconvex candidate $\Wmp$. Energies with a non-convex isochoric term $h(t)$ cannot be reduced to a single candidate ($\Wmm\,,\;\Ws\,,\;\Wn\,,\cdots$) and it remains to show whether or not there exists a non-polyconvex representative.}\label{fig:morreySummary}
\end{figure}
\restoregeometry
%
%
%
%
%
\section{Connection to the Burkholder functional}\label{sec:Burkholder}
Recall that quasiconvexity of $\Wmp(F)$ would imply quasiconvexity for all rank-one convex energies with an additive volumetric-isochoric split and a convex isochoric part, while proving the opposite would lead to an example of Morrey's conjecture.
\begin{definition}
	For an energy function $W\col\GLp(n)\to\R$, we define its \emph{Shield transformation} \cite{shield1967inverse,ball1977constitutive} as
	\begin{equation}\label{eq:shieldTransformation}
		W^\#\col\GLp(n)\to\R\,,\qquad W^\#(F)=\det(F)\cdot W(F\inv)\,.
	\end{equation}
\end{definition}
\begin{lemma}\label{lemma:shieldTransformationProperties}
	Let $W,W_1,W_2\col\GLpn\to\R$. Then
	\begin{itemize}
		\item[i)] $W^{\#\#}=W$,
		\item[ii)] $W_1\leq W_2$ if and only if $W_1^\#\leq W_2^\#$,
		\item[iii)] $W$ is rank-one convex if and only if $W^\#$ is rank-one convex,
		\item[iv)] $W$ is $C_0^\infty$-quasiconvex if and only if $W^\#$ is $C_0^\infty$-quasiconvex,\footnotemark
		\item[v)] $W$ is polyconvex if and only if $W^\#$ is polyconvex.
	\end{itemize}
\end{lemma}
\begin{proof}
	The equality
	\[
		W^{\#\#}(F) = \det(F)\cdot W^\#(F\inv) = \det(F)\cdot [\det(F\inv)\cdot W((F\inv)\inv)] = W(F)
	\]
	directly establishes i). Now, if $W_1\leq W_2$, i.e.\ $W_1(F)\leq W_2(F)$ for all $F\in\GLpn$, then
	\[
		W_1^\#(F) = \det(F)\cdot W_1(F\inv) \leq \det(F)\cdot W_2(F\inv) = W_2^\#(F)
	\]
	which, combined with i), establishes ii). Finally, the invariance of rank-one convexity, quasiconvexity\footnotemark[\value{footnote}] and polyconvexity under the Shield transformation are well known \cite[Theorem 2.6]{ball1977constitutive}, and the equivalence can thus be inferred from i).
	\footnotetext{To the best of our knowledge, the equivalence between the quasiconvexity of $W$ and $W^\#$ has only been shown if non-smooth variations $\vartheta\not\in C_0^\infty(\Omega)$ are excluded.}
\end{proof}
Note that (classical) convexity of $W$ is not equivalent to the convexity of $W^\#$ for dimension $n\geq 2$; for example, the constant mapping $W\equiv1$ is convex, whereas $W^\#(F)=\det F$ is not. Furthermore, in the one-dimensional case, the invariance of the convexity properties can be expressed as the equivalence
\[
	f\col\Rp\to\R\quad\text{is convex}\qquad\iff\qquad f^\#\col\Rp\to\R\,,\quad f^\#(t) = t\cdot f\left(\frac1t\right)\quad \text{is convex}\,,
\]
which is related to the study of so-called \emph{reciprocally convex functions} \cite[Lemma 2.2]{merkle2004reciprocally} and, for \mbox{$f\in C^2(\Rp)$}, follows directly from the equality
\[
	\frac{\mathrm{d}^2}{\mathrm{d}t^2}\left[t\cdot f\left(\frac1t\right)\right]	= \ddt\left[ f\left(\frac1t\right) - \frac1t\.f'\left(\frac1t\right) \right]
	= \frac{1}{t^3}\.f''\left(\frac1t\right),\qquad t>0\,.
\]

Applying now the Shield transformation, we show a surprising connection between our candidate $\Wmp(F)$ and the work of Iwaniec in the field of complex analysis and the so-called Burkholder functional. Using the connection $\C^2\cong\R^{2\times2}$ we arrive at a much more general settings of functions $W\col\R^{2\times2}\to\R$ without a volumetric-isochoric split defined on the linear space $\R^{2\times2}$ instead $\GLp(2)$.
%
%
%
\subsection{The Burkholder functional}
In the following, we sketch an alternative approach to obtain the energy candidate
\begin{align}
	\Wmp(F)=\frac{\lambdamax}{\lambdamin}-\log\frac{\lambdamax}{\lambdamin}+\log\det F=\frac{\lambdamax}{\lambdamin}+ 2\.\log\lambdamin\,.
\end{align}
\begin{theorem}[Proposition 5.1 \cite{iwaniec2002nonlinear,iwaniec1999nonlinear}]\label{theorem:burkholder}
	The function
	\setlength{\belowdisplayskip}{5pt}
	\begin{align}
		W_{\mathrm{IW}}^-\col\Rnn\to\R\,,\qquad W_{\mathrm{IW}}^-(F)=\underbrace{\abs{1-\frac{n}{p}}}_{\equalscolon c}\,\opnorm{F}^p-\opnorm{F}^{p-n}\det F
	\end{align}
	is rank-one convex for all $p\geq\frac{n}{2}$. Furthermore, the factor $c=\abs{1-\frac{n}{p}}$ is the smallest possible constant for which this statement is true.
\end{theorem}
Here $\opnorm{F}\colonequals\sup_{\norm{\xi}=1}\norm{F\.\xi}_{\R^2}=\lmax$ denotes the operator norm (i.e.\ the largest singular value) of $F$. The function $W_{\rm IW}^-$ is homogeneous of degreee $p$, i.e.\ $W(\alpha\.F)=\alpha^p\.W(F)$, but it is a non-polynomial function due to its dependence\footnote{Note that $W_{\rm IW}^-(F)$ is not differentiable at $F=\lambda\.R$, $\lambda\in(0,\infty)$, $R\in\OO(2)$ due to the operator-norm at non-single singular values.} on $\opnorm{F}$. Moreover, the function is not isotropic but only hemitropic\footnote{Functions $W\col\Rnn\to\R$ that are right-invariant w.r.t\ the proper orthogonal group $\SO(n)$ are called \emph{hemitropic}, while functions that are right-invariance w.r.t.\ the larger group of orthogonal matrices $\On$ are called \emph{isotropic}.} (not right-$\OO(2)$-invariant).

We continue with the planar case ($n=2$) and restrict the investigation to $p\geq 2$:
\begin{align}
	W_{\mathrm{IW}}^-(F)&=\abs{\underbrace{1-\frac{2}{p}}_{>0}}\,\opnorm{F}^p-\opnorm{F}^{p-2}\det F=\left[-\det F+\left(1-\frac{2}{p}\right)\opnorm{F}^2\right]\opnorm{F}^{p-2}\,.
\end{align}
Let us define
\begin{align}
	 B_p(F)\colonequals\frac{p}{2}W_{\rm IW}^-(F)&=-\left[\frac{p}{2}\.\det F+\left(1-\frac{p}{2}\right)\opnorm{F}^2\right]\opnorm{F}^{p-2}\,,
\end{align}
where
\begin{align}
	-B_p(F)&=\left(\frac{p}{2}\.\det F+\left(1-\frac{p}{2}\right)\lmax^2\right)\lmax^{p-2}=\frac{p}{2}\.\det F\.\lmax^{p-2}+\frac{2-p}{2}\.\lmax^p\label{eq:burkhoder}\\
	&=\frac{p}{2}\.\det F\.\left(\frac{\norm{F}^2+\sqrt{\norm{F}^4-4\.(\det F)^2}}{2}\right)^\frac{p-2}{2}+\frac{2-p}{2}\.\left(\frac{\norm{F}^2+\sqrt{\norm{F}^4-4\.(\det F)^2}}{2}\right)^\frac{p}{2}\notag
\end{align} is also known as the \emph{Burkholder functional} \cite{burkholder1988sharp,astala2012quasiconformal}. It is already established that the function $B_p$ is rank-one convex and satisfies a quasiconvexity property\footnote{In the literature, $B_p(F)$ is usually defined with a reversed sign and it is investigated with respect to rank-one concavity and quasiconcavity instead \cite{astala2012quasiconformal}.} at the identity $F=\id$ in the case of smooth non-positive integrands \cite[Theorem 1.2]{astala2012quasiconformal}, namely if $\vartheta\subset C_0^\infty(\Omega;\R^2)$ and $B_p(\id+\nabla\vartheta(x))\leq0$ for all $x\in\Omega$, then
\begin{equation}
	\int_\Omega B_p(\id+\nabla\vartheta(x))\,\dx\geq\int_\Omega B_p(\id)\,\dx=-\abs\Omega\,,\qquad p\geq2\,.
\end{equation}
\begin{lemma}[\cite{burkholder1988sharp}]\label{lemma:BurkholderRankOneConvex}
	The function $B_p\col\R^{2\times2}\to\R$ with
	\[
		 B_p(F)=-\left[\frac{p}{2}\.\det F+\left(1-\frac{p}{2}\right)\opnorm{F}^2\right]\opnorm{F}^{p-2}
	\]	
	is rank-one convex.
\end{lemma}
\begin{proof}
	For the convenience of the reader, we add the proof from \cite{burkholder1988sharp}: We use the connection to complex analysis $\R^{2\times2}\cong\C^2$ (cf.\ Appendix \ref{sec:complexAnalysis}) where we identify
	\begin{equation}
		F(z,w)\colonequals\matr{\Re z+\Re w & \Im w -\Im z\\ \Im z+\Im w & \Re z-\Re w}
	\end{equation}
	and transform the Burkholder function $B_p(F)$ to a function $L_p\col\C\times\C\to\R$ with
	\begin{equation}
		L_p(z,w)=\left(\abs z-(p-1)\abs w\right)\left(\abs z +\abs w\right)^{p-1}.
	\end{equation}
	Then $B_p(F(z,w))=-L_p(z,w)$ and, for rank-one matrices $\xi\otimes\eta$ with $\xi,\eta\in\R^2$, we find
	\begin{equation}
		B_p(F(z,w)+t\.\xi\otimes\eta)=-L_p(z+t\.h,w+t\.k)\,,
	\end{equation}
	where
	\begin{align*}
		\matr{\xi_1\eta_1 & \xi_1\eta_2\\ \xi_2\eta_1 & \xi_1\eta_2}=\matr{\Re h+\Re k & \Im k -\Im h\\ \Im h+\Im k & \Re h-\Re k}\quad\iff\quad\begin{matrix}
		\Re h=\frac12\left(\xi_1\eta_1+\xi_2\eta_2\right),&\Re k=\frac12\left(\xi_1\eta_1-\xi_2\eta_2\right),\\
		\Im h=\frac12\left(\xi_2\eta_1-\xi_1\eta_2\right),&\Im k=\frac12\left(\xi_1\eta_2+\xi_2\eta_1\right).
		\end{matrix}
	\end{align*}
	In addition, for rank-one matrices with this definition,
	\begin{align}
		4\.\abs h^2&=\abs{2\.\Re h}^2+\abs{2\.\Im h}^2=(\xi_1\eta_1+\xi_2\eta_2)^2+(\xi_2\eta_1-\xi_1\eta_2)^2\notag\\
		&=\xi_1^2\eta_1^2+2\.\xi_1\eta_1\xi_2\eta_2+\xi_2^2\eta_2^2+\xi_2^2\eta_1^2-2\.\xi_2\eta_1\xi_1\eta_2+\xi_1^2\eta_2^2=\xi_1^2\eta_1^2+\xi_1^2\eta_2^2+\xi_2^2\eta_1^2+\xi_2^2\eta_2^2\\
		&=\xi_1^2\eta_1^2-2\.\xi_1\eta_1\xi_2\eta_2+\xi_2^2\eta_2^2+\xi_1^2\eta_2^2+2\.\xi_1\eta_2\xi_2\eta_1+\xi_2^2\eta_1^2=(\xi_1\eta_1-\xi_2\eta_2)^2+(\xi_1\eta_2+\xi_2\eta_1)^2\notag\\
		&=\abs{2\.\Re k}^2+\abs{2\.\Im k}^2=4\.\abs k^2\qquad\implies\qquad\abs h=\abs k.\notag
	\end{align}
	Now the rank-one convexity of $B_p(F)$ follows from the concavity of $-L_p(z+t\.h,w+t\.k)$ (cf.\ \cite[eq.\ (1.14)]{burkholder1988sharp}), i.e.\ the observation that for all $z,w,h,k\in\C$ with $\abs k\leq\abs h$ the mapping
	\begin{equation}
		t\mapsto-L_p(z+t\.h,w+t\.k)
	\end{equation}
	is concave on $\R$.
\end{proof}
The function of Burkholder plays an important role in the martingale study of the Beurling-Ahlfors operator (cf.\ Appendix \ref{sec:complexAnalysis}) where several open questions (in particular an important conjecture of Iwaniec \cite{iwaniec1993quasiregular}) could be answered by showing that $B_p(F)$ is quasiconvex.

In case of $p=2$ the energy function $B_2(F)$ is a \emph{Null-Lagrangian} since
\begin{equation}
	 B_2(F)=-\left[\det F+\left(1-1\right)\opnorm{F}^2\right]\opnorm{F}^{2-2}=-\det F\,.
\end{equation}
Therefore, the expression $B_p(F)- B_2(F)$ has the same convexity behavior as $B_p$ itself and we can formally compute the derivative with respect to $p$. We define
\begin{align}
	 B_\star(F)&\colonequals\lim_{p\searrow 2}\frac{ B_p(F)- B_2(F)}{p-2}=\ddp B_p(F)|_{p=2}=\ddp\left[\left(-\frac{p}{2}\det F+\left(\frac{p}{2}-1\right)\opnorm{F}^2\right)e^{(p-2)\.\log\opnorm{F}}\right]_{p=2}\notag\\
	&=-\frac{1}{2}\.\det F+\frac{1}{2}\.\opnorm{F}^2+\left(-\det F+0\right)\log\opnorm{F}=-\frac{1}{2}\left(1+2\.\log\opnorm{F}\right)\det F+\frac{1}{2}\.\opnorm{F}^2\notag\\
	&=-\frac{1}{2}\left(1+\log\big(\opnorm{F}^2\big)\right)\det F+\frac{1}{2}\.\opnorm{F}^2.
\end{align}
The function $B_\star$ remains globally rank-one convex and quasiconvex at the identity.\footnote{For the quasiconvexity of $B_\star$ we use that $\lim_{k\to\infty}\int_\Omega W_k(\nabla\varphi(x))\,\dx=\int_\Omega\lim_{k\to\infty}W_k(\nabla\varphi(x))\,\dx=\int_\Omega B_\star(\nabla\varphi(x))\,\dx$ with $W_k(\nabla\varphi(x))\colonequals-\frac1k\bigl(B_{2+\frac1k}(\nabla\varphi(x))- B_2(\nabla\varphi(x))\bigr)$ holds because for every $\nabla\varphi\in\R^{2\times 2}$ the expression $\abs{W_k(\nabla\varphi(x))}$ is bounded for all $x\in\R^2$ and $k\in\N$.} If we restrict $B_\star$ to $\GLp(2)$, we may now calculate its Shield transformation 
\begin{align}
	B^\#_\star(F)&=\det F\cdot B_\star(F\inv)=\det F\left[-\frac{1}{2}\left(1+\log\big(\opnorm{F\inv}^2\big)\right)\det(F\inv)+\frac{1}{2}\.\opnorm{F\inv}^2\right]\notag\\
	&=-\frac{1}{2}-\log\opnorm{F\inv}+\frac{1}{2}\.\det F\.\opnorm{F\inv}^2\notag\\
	&=-\frac{1}{2}-\log\left(\frac{\opnorm{F}}{\det F}\right)+\frac{1}{2}\.\det F\.\frac{\opnorm{F}^2}{(\det F)^2}\label{eq:shieldBurkhoder}\\
	&=-\frac{1}{2}-\log\opnorm{F}+\log\det F+\frac{1}{2}\.\frac{\opnorm{F}^2}{\det F}\,,\notag
\end{align}
where we have used that in the two-dimensional case,
\begin{align*}
	\opnorm{F\inv}=\frac{1}{\lmin}=\frac{\lmax}{\lmin\lmax}=\frac{\opnorm{F}}{\det F}\,.
\end{align*}
Omitting the additive constant $-\frac{1}{2}$ in equation \eqref{eq:shieldBurkhoder}, which has no effect on the convexity behavior, we arrive at
\begin{align}
	B^\#_\star(F)+\frac{1}{2}&=\frac{1}{2}\.\frac{\opnorm{F}^2}{\det F}-\log\opnorm{F}+\log\det F\\
	&=\frac{1}{2}\.\frac{\lmax^2}{\lmax\lmin}-\log\lmax+\log(\lmax\lmin)=\frac{1}{2}\.\frac{\lmax}{\lmin}+\log\lmin=\frac{1}{2}\Wmp(F)\,.\notag
\end{align}
In particular, Lemma \ref{lemma:shieldTransformationProperties} and Theorem \ref{theorem:burkholder} also imply that $\Wmp(F)$ is rank-one convex.

An overview of the relation between our energy candidate $\Wmp(F)$ and the Burkholder functional $B_p(F)$ is visualizied in Figure \ref{fig:burkholderConnection}. Again, while we try to find an example to settle Morrey's question, i.e.\ show that our energy function $\Wmp(F)$ is \textbf{not} quasiconvex, in the topic of martingale study, it is sought to prove that $B_p(F)$ is quasiconvex to positively affirm some estimates in function theory, see Appendix \ref{sec:complexAnalysis} for more information.
\begin{figure}[h!]
	\centering
	\begin{tikzpicture}[->,>=stealth',implies/.style={thick,double,double equal sign distance,-implies},iff/.style={thick,double,double equal sign distance,implies-implies}]
	 		\node[state] (1){
	 			\begin{tabular}{c}
	 				Iwaniec/Burkholder approach:\\
	 				$\displaystyle B_p\col\R^{2\times 2}\to\R$
	 			\end{tabular}
	 		};
		 	\node[state, right of=1, node distance=9cm] (2){
	  			\begin{tabular}{c}
	 				 Volumetric-isochoric split approach:\\
	 				$\displaystyle W\col\GLp(2)\to\R\,,\; W(F)=h\Bigl(\frac{\lambda_1}{\lambda_2}\Bigr)+f(\det F)$
	 			\end{tabular}
	 		};
	 		\node[state, below of=1,node distance =3cm](3){
	 			\begin{tabular}{c}
	 				$\displaystyle B_p(F)=-\left[\frac{p}{2}\.\det F+\left(1-\frac{p}{2}\right)\opnorm{F}^2\right]\opnorm{F}^{p-2}$\\
	 				the \enquote{heart of the matter} \cite{burkholder1988sharp}
	 			\end{tabular}
	 		};
	 		\node[state, below of=2, node distance=3cm, align=center] (4){
				\begin{tabular}{c}
	 				$\displaystyle \Wmp(F)=\frac{\lmax}{\lmin}-\log\frac{\lmax}{\lmin}+\log(\det F)$\\
	 				least rank-one convex candidate
	 			\end{tabular}				  			
	 		};
	 		\node[state, below of=3,node distance =3cm](5){
	 			\begin{tabular}{c}
	 				$\displaystyle B_\star(F)\colonequals\lim_{p\searrow 2}\frac{ B_p(F)- B_2(F)}{p-2}=\ddp B_p(F)\bigr|_{p=2}$\\
	 				\hspace{0.9cm}$\displaystyle =-\frac{1}{2}\left(1+\log\big(\opnorm{F}^2\big)\right)\det F+\frac{1}{2}\.\opnorm{F}^2$
	 			\end{tabular}
	 		};
	 		\node[state, below of=4, node distance=3cm, align=center] (6){
				$\displaystyle B_\star^\#(F)\colonequals\frac{1}{2}\bigl(\Wmp(F)-1\bigr)$				  			
	 		};
	 		\path[ultra thick]
	 			(1) edge (3)
	 			(2) edge (4)
	 			(3) edge node[midway,right]{$B_2(F)$ is Null-Lagrangian} (5)
	 			(4) edge (6)
	 			(5) edge (6)
	 			(6) edge node[midway,above]{Shield} node[midway,below]{transformation} (5)
	 		;
	\end{tikzpicture}
	\caption{Connection between the Burkholder functional $B_p(F)$ and our energy candidate $\Wmp(F)$ as a chance to check Morrey's conjecture, where $\opnorm{F}=\sup_{\norm{\xi}=1}\norm{F\.\xi}_{\R^2}=\lmax$ denotes the operator norm (i.e.\ the largest singular value) of $F$.}
	\label{fig:burkholderConnection}
\end{figure}
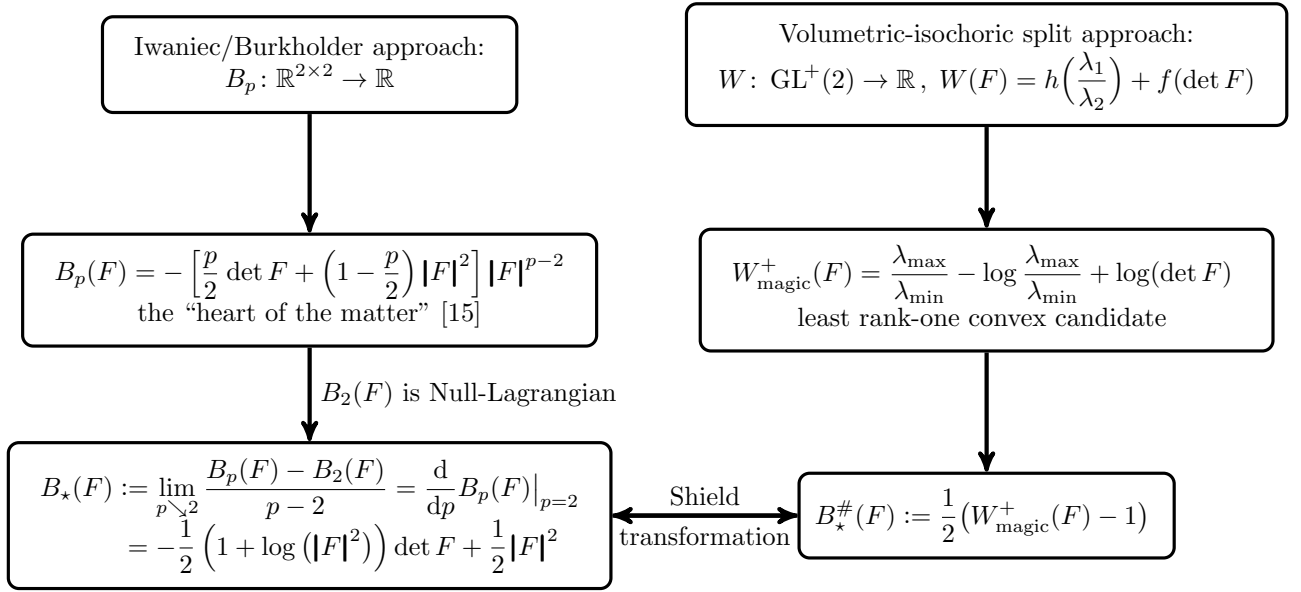
%
%
%
%
%
\section{Radially symmetric deformations}
Radial deformations play an important role for planar quasiconvexity. For isotropic materials it seems reasonable that any global minimizer under radial boundary conditions is radial, but such a symmetry statement may cease to hold for non-convex nonlinear problems\footnote{Let $\Omega$ be the unit ball in $\R^n$. Then the minimization problem of the non-convex functional
\begin{equation}
	R(v)=\int_\Omega\frac{1}{1+\norm{\nabla v(x)}^2}\,\dx\label{eq:nonRadialMinimizer}
\end{equation}
has neither a radial nor a unique solution in the class $C_M\colonequals\left\{v\in W^{1,\infty}_{\rm loc}(\Omega;\R)\;|\;0\leq v(x)\leq M\;\forall\;x\in\Omega\,,\; v\text{ is concave}\right\}$.} \cite{brock1996symmetry}. A deformation is radial if there exists  a function $v\col[0,\infty)\to\R$ such that
\begin{equation}
	\varphi(x)=v(\norm x)\.\frac{x}{\norm x}\qquad\text{with}\quad v(0)=0\,.\label{eq:radialMappings}
\end{equation}
We compute
\begin{align}
	\nabla\varphi(x)&=\frac{v(\norm{x})}{\norm{x}}\.\id_2+\left(\frac{v'(\norm{x})\.\frac{x}{\norm x}\.\norm x-v(\norm x)\.\frac{x}{\norm x}}{\norm x^2}\right)\otimes x\notag\\
	&=\frac{v(\norm{x})}{\norm{x}}\.\id_2+\left(v'(\norm{x})-\frac{v(\norm{x})}{\norm{x}}\right)\frac{1}{{\norm{x}}^2}\.x\otimes x
\end{align}
and identify the corresponding singular values
\begin{align}
	\lambda_1=v'(\norm{x})\,,\qquad\lambda_2=\frac{v(\norm{x})}{\norm{x}}\,.
\end{align}
\begin{lemma}[\cite{ball1987does}]\label{lem:radialSeparateConvex}
	Let $W\col\GLp(2)\to\R$ be an isotropic, objective and separately convex energy function and $g(\lambda_1,\lambda_2)=W(F)$ with $\lambda_1,\lambda_2$ as the singular values of $F$. If $F_0\in\GLp(2)$ is a homogeneous radial mapping, i.e.\ $F_0=\lambda\.\id$ with $\lambda\in\Rp$, then
	\begin{equation}
		\int_{B_1(0)}W(\nabla\varphi(x))\,\dx\geq\int_{B_1(0)}W(F_0)\,\dx=\pi\.g(\lambda,\lambda)
	\end{equation}
	for all radial deformations $\varphi$ of the form \eqref{eq:radialMappings} with $\varphi|_{\partial B_1(0)}(x)=F_0\.x$.
\end{lemma}
\begin{proof}
	For the homogeneous radial deformation gradient of the form $F_0=\lambda\.\id$ we compute
	\begin{align}
		\int_{B_1(0)}W(\lambda\.\id)\.\dx_1\dx_2=2\pi\.\int_0^1 r\. g\left(\lambda,\lambda\right)\dr=2\pi\. g\left(\lambda,\lambda\right)\left[\frac{r^2}{2}\right]_0^1=\pi\. g(\lambda,\lambda)\,.
	\end{align}
	For an arbitrary radial deformation $\varphi(x)=v(\norm x)\.\frac{x}{\norm x}=g\left(v'(\norm{x}),\frac{v(\norm{x})}{\norm{x}}\right)$ we write
	\[
		\int_{B_1(0)}W(\nabla\varphi(x_1,x_2))\,\dx_1\dx_2=\int_{B_1(0)}g(\lambda_1,\lambda_2)\,\dx_1\dx_2=\int_0^{2\pi}\int_0^1 g\left(v',\frac{v}{r}\right)r\.\dr\.\dth=2\pi\.\int_0^1 r\. g\left(v',\frac{v}{r}\right)\dr\,.
	\]
	Separate convexity of $g$ in the first argument ensures
	\[
		g(x_2,y)\geq g(x_1,y)+(x_2-x_1)\.g_x(x_1,y)\qquad \text{for all}\quad x_1,x_2,y\in\R\,.
	\]
	Setting symbolically $x_1=y=\frac{v}{r}$ and $x_2=v'$ implies
	\begin{align}
		 g\left(v',\frac{v}{r}\right)&\geq  g\left(\frac{v}{r},\frac{v}{r}\right)+\left(v'-\frac{v}{r}\right) g_x\left(\frac{v}{r},\frac{v}{r}\right).
	\end{align}
	Thus we compute
		\begin{align}
		\int_{B_1(0)}W(\nabla\varphi(x_1,x_2))\,\dx&=2\pi\.\int_0^1 r\. g\left(v',\frac{v}{r}\right)\dr\geq 2\pi\.\int_0^1 r\left[ g\left(\frac{v}{r},\frac{v}{r}\right)+\left(v'-\frac{v}{r}\right)g_x\left(\frac{v}{r},\frac{v}{r}\right)\right]\dr\\
		&\overset{\mathclap{(*)}}{=}2\pi\.\int_0^1 \ddr\left[\frac{r^2}{2}\. g\left(\frac{v(r)}{r},\frac{v(r)}{r}\right)\right]\dr=\pi\left[r^2\. g\left(\frac{v(r)}{r},\frac{v(r)}{r}\right)\right]_0^1=\pi\. g(\lambda,\lambda)\notag
	\end{align}
	with $\lambda=v(1)$ and $v(0)=0$. Equality $(*)$ follows from
	\begin{align}
		\ddr\left[\frac{r^2}{2}\. g\left(\frac{v(r)}{r},\frac{v(r)}{r}\right)\right]&=r\. g\left(\frac{v(r)}{r},\frac{v(r)}{r}\right)+r^2\.g_x\left(\frac{v(r)}{r},\frac{v(r)}{r}\right)\left(\frac{v'(r)}{r}-\frac{v(r)}{r^2}\right)\notag\\
		&=r\left[ g\left(\frac{v(r)}{r},\frac{v(r)}{r}\right)+\left(v'(r)-\frac{v(r)}{r}\right)g_x\left(\frac{v(r)}{r},\frac{v(r)}{r}\right)\right],
	\end{align}
	where we used the equality $g_x(\lambda,\lambda)=g_y(\lambda,\lambda)$; note that $g(\lambda_1,\lambda_2)=g(\lambda_2,\lambda_1)$ due to the isotropy of $W$. Overall,
	\[
		\int_{B_1(0)}W(\nabla\varphi(x_1,x_2))\,\dx\geq\pi\. g(\lambda,\lambda)=\int_{B_1(0)}W(F_0)\,\dx\,.\qedhere
	\]
\end{proof}
%
%
%
\subsection{Expanding and contracting deformations}\label{sec:contractingDeformations}
In general, non-trivial radial deformations do not have the same energy value as the homogeneous deformation $\varphi_0(x)=\lambda\.x$, especially if the energy is strictly separately convex. However, there exist non-trivial examples for the Burkholder energy $B_p(F)$ as well as for the energy $\Wmp(F)$ for which equality holds. In order to construct such deformations, we first consider the general class of radial functions with $v(0)=0$ and $v(R)=R$ which keep the center point and exterior radius constant. More specifically, we consider the subclasses of expanding and contracting functions
\begin{align}
	\mathcal{V}_R\colonequals&\left\{v\in C^1([0,R])\,\big|\,v(0)=0\,,\;v(R)=R\,,\; \frac{v(r)}{r}\geq v'(r)\geq 0\quad\forall\; r\in[0,R]\right\},\tag{expanding}\\
	\mathcal{V}_R\inv=&\left\{v\in C^1([0,R])\,\big|\,v(0)=0\,,\;v(R)=R\,,\; v'(r)\geq \frac{v(r)}{r}\geq 0\quad\forall\; r\in[0,R]\right\},\tag{contracting}
\end{align}
for which the order of the singular values $\lmin$ and $\lmax$ remains fixed.\footnote{The condition $\frac{v(r)}{r}\geq v'(r)$ is not trivial and only allows radial functions which inflate all disks with center in zero, i.e.\ $v(r)\geq r$ for all $r\in[0,R]$. Note that if there exists a point $r_0$ with $v(r_0)<r_0$, then $v'(r_0)\leq\frac{v(r_0)}{r_0}<1$, which would imply that $v(r)<r$ for all $r\geq r_0$ and thus $v(R)<R$.} The class $\mathcal V_R$ of expanding functions can be described by radial functions for which the gradient of the tangent in a point $(r,v(r))$ is smaller then the gradient of the secant of the origin with $(r,v(r))$ for any $r\in[0,R]$ as shown in Figure \ref{fig:radialMappings}.

In contrast, the class $\mathcal V_R\inv$ of contracting functions contains any radial functions $v$ such that the gradient of the tangent in a point $(r,v(r))$ is larger than the gradient of the secant of the origin with $(r,v(r))$ for any $r\in[0,R]$ as shown in Figure \ref{fig:radialMappings}. The connection between $\mathcal{V}_R$ and $\mathcal V_R\inv$ is given by the inverse function theorem: For an arbitrary $v\in\mathcal V_R$ with $r\mapsto v(r)$ and its inverse function $v\inv\in\mathcal V_R\inv$ with $t\mapsto v\inv(t)$ we find
\begin{align}
	v\inv(t)=\frac{1}{v'\big(v\inv(t)\big)}\geq\frac{v\inv(t)}{t}\qquad\iff\qquad\frac{t}{v\inv(t)}\geq v'\big(v\inv(t)\big)\qquad\iff\qquad\frac{v(r)}{r}\geq v'(r)\,.
\end{align}
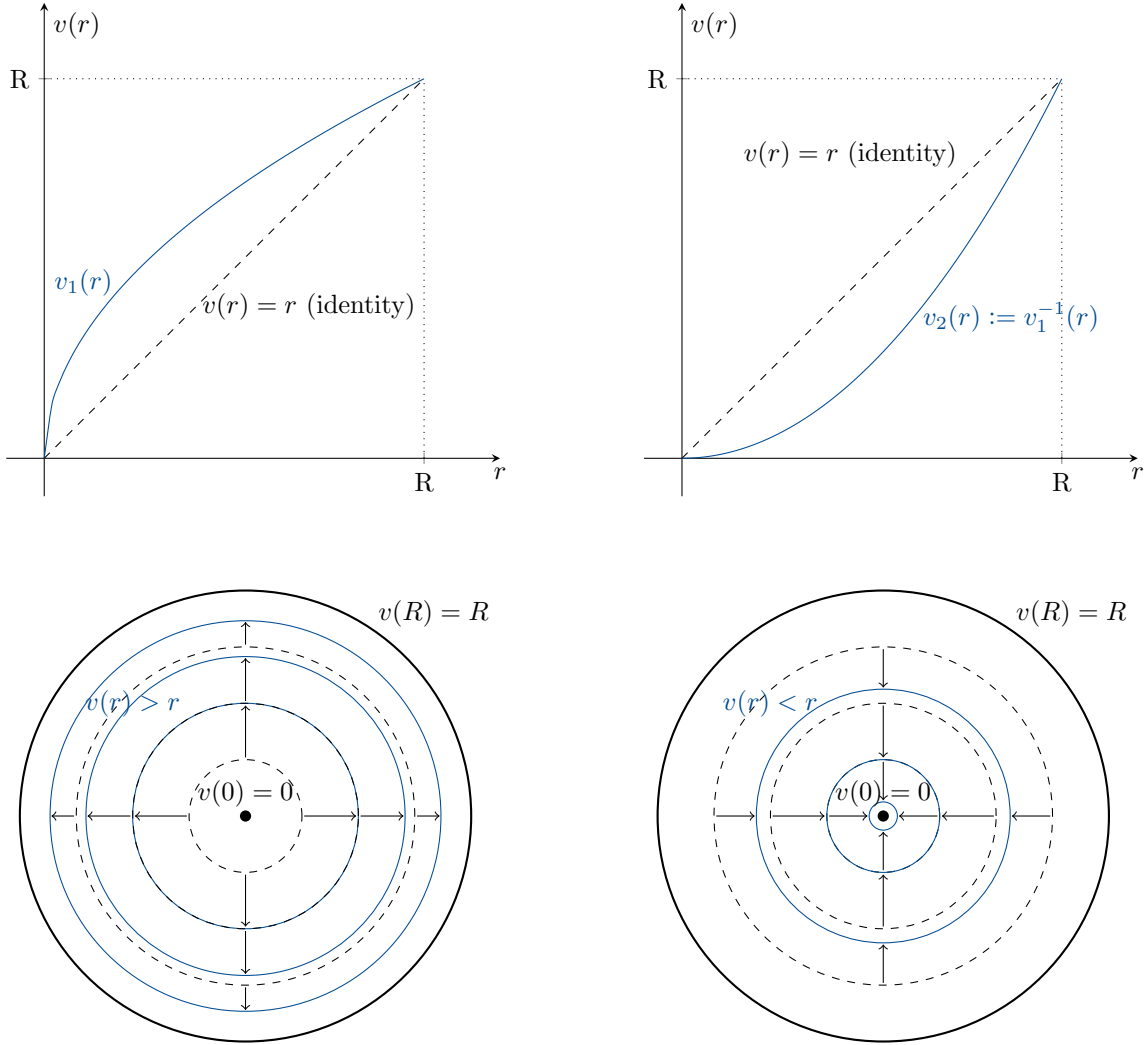
\begin{figure}[h!]
	\centering
	\begin{minipage}[t]{.49\linewidth}
	\centering
	\begin{tikzpicture}
  		\begin{axis}[
        axis x line=middle,axis y line=middle,
        x label style={at={(current axis.right of origin)},anchor=north, below},
        xlabel=$r$, ylabel=$v(r)$,
        xmin=-0.1, xmax=1.2,
        ymin=-0.1, ymax=1.2,
        width=1\linewidth,
        height=1\linewidth,
        ytick=\empty,
        xtick={1},
        xticklabels={R},
        ytick={1},
        yticklabels={R}
        ]
        \addplot[black, dotted] coordinates {(0, 1) (1, 1)};
        \addplot[black, dotted] coordinates {(1, 0) (1, 1)};
        \addplot[black, dashed][domain=0:1, samples=\sample]{x} node[pos=0.4,xshift=1.5cm] {$v(r)=r$ (identity)};
        \addplot[udedarkblue, smooth][domain=0:1, samples=\sample]{sqrt(x)} node[pos=0.3,above,xshift=-0.3cm] {$v_1(r)$};
       \end{axis}
    \end{tikzpicture}
	\end{minipage}
	\hfill
	\begin{minipage}[t]{.49\linewidth}
	\centering
	\begin{tikzpicture}
  		\begin{axis}[
        axis x line=middle,axis y line=middle,
        x label style={at={(current axis.right of origin)},anchor=north, below},
        xlabel=$r$, ylabel=$v(r)$,
        xmin=-0.1, xmax=1.2,
        ymin=-0.1, ymax=1.2,
        width=1\linewidth,
        height=1\linewidth,
        ytick=\empty,
        xtick={1},
        xticklabels={R},
        ytick={1},
        yticklabels={R}
        ]
        \addplot[black, dotted] coordinates {(0, 1) (1, 1)};
        \addplot[black, dotted] coordinates {(1, 0) (1, 1)};
        \addplot[black, dashed][domain=0:1, samples=\sample]{x} node[pos=0.8,xshift=-1.8cm] {$v(r)=r$ (identity)};
        \addplot[udedarkblue, smooth][domain=0:1, samples=\sample]{x^2} node[pos=0.5,right] {$v_2(r)\colonequals v_1\inv(r)$};
       \end{axis}
    \end{tikzpicture}
	\end{minipage}
	$\phantom{+}$\\[2em]
	\begin{minipage}[t]{.49\linewidth}
	\centering
	\begin{tikzpicture}
		\def\rad{3cm};
  		\coordinate (O) at (0,0);
  		\coordinate (1) at (2.5,2.7);
  		\coordinate (2) at (-1.5,1.5);
		\draw[thick] (O) node[circle,inner sep=1.5pt,fill] {} circle [radius=\rad];
		\draw[dashed] (O) circle [radius=0.25*\rad];
		\draw[udedarkblue] (O) circle [radius=sqrt(0.25)*\rad];
		\draw[->] (0.26*\rad,0) -- (0.49*\rad,0);
		\draw[->] (-0.26*\rad,0) -- (-0.49*\rad,0);
		\draw[->] (0,0.26*\rad) -- (0,0.49*\rad);
		\draw[->] (0,-0.26*\rad) -- (0,-0.49*\rad);
		\draw[dashed] (O) circle [radius=0.5*\rad];
		\draw[udedarkblue] (O) circle [radius=sqrt(0.5)*\rad];
		\draw[->] (0.51*\rad,0) -- (0.7*\rad,0);
		\draw[->] (-0.51*\rad,0) -- (-0.7*\rad,0);
		\draw[->] (0,0.51*\rad) -- (0,0.7*\rad);
		\draw[->] (0,-0.51*\rad) -- (0,-0.7*\rad);
		\draw[dashed] (O) circle [radius=0.75*\rad];
		\draw[udedarkblue] (O) circle [radius=sqrt(0.75)*\rad];
		\draw[->] (0.76*\rad,0) -- (0.86*\rad,0);
		\draw[->] (-0.76*\rad,0) -- (-0.86*\rad,0);
		\draw[->] (0,0.76*\rad) -- (0,0.86*\rad);
		\draw[->] (0,-0.76*\rad) -- (0,-0.86*\rad);
		\draw (O) node[above] {$v(0)=0$};
		\draw (1) node {$v(R)=R$};
		\draw[udedarkblue] (2) node {$v(r)>r$};
    \end{tikzpicture}
	\end{minipage}
	\hfill
	\begin{minipage}[t]{.49\linewidth}
	\centering
	\begin{tikzpicture}
  		\def\rad{3cm};
  		\coordinate (O) at (0,0);
  		\coordinate (1) at (2.5,2.7);
  		\coordinate (2) at (-1.5,1.5);
		\draw[thick] (O) node[circle,inner sep=1.5pt,fill] {} circle [radius=\rad];
		\draw[dashed] (O) circle [radius=0.25*\rad];
		\draw[udedarkblue] (O) circle [radius=0.25^2*\rad];
		\draw[->] (0.24*\rad,0) -- (0.07*\rad,0);
		\draw[->] (-0.24*\rad,0) -- (-0.07*\rad,0);
		\draw[->] (0,0.24*\rad) -- (0,0.07*\rad);
		\draw[->] (0,-0.24*\rad) -- (0,-0.07*\rad);
		\draw[dashed] (O) circle [radius=0.5*\rad];
		\draw[udedarkblue] (O) circle [radius=0.5^2*\rad];
		\draw[->] (0.49*\rad,0) -- (0.26*\rad,0);
		\draw[->] (-0.49*\rad,0) -- (-0.26*\rad,0);
		\draw[->] (0,0.49*\rad) -- (0,0.26*\rad);
		\draw[->] (0,-0.49*\rad) -- (0,-0.26*\rad);
		\draw[dashed] (O) circle [radius=0.75*\rad];
		\draw[udedarkblue] (O) circle [radius=0.75^2*\rad];
		\draw[->] (0.74*\rad,0) -- (0.57*\rad,0);
		\draw[->] (-0.74*\rad,0) -- (-0.57*\rad,0);
		\draw[->] (0,0.74*\rad) -- (0,0.57*\rad);
		\draw[->] (0,-0.74*\rad) -- (0,-0.57*\rad);
		\draw (O) node[above] {$v(0)=0$};
		\draw (1) node {$v(R)=R$};
		\draw[udedarkblue] (2) node {$v(r)<r$};
    \end{tikzpicture}
	\end{minipage}
  \caption{Left: Visualization of an expanding radial function $v_1\in\mathcal{V}_R$ with $v_1(r)>r$. Right: Visualization of an contracting radial function $v_2\in\mathcal{V}_R\inv$ with $v_2(r)<r$ which is the inverse of $v_1$.}\label{fig:radialMappings}
\end{figure}
\begin{lemma}[\cite{astala2012quasiconformal}]\label{lemma:burkholderRadial}
	The energy value of the Burkholder energy is constant on the class of \emph{expanding} functions $\mathcal V_R$, i.e.\
	\begin{equation}
		\int_{\BRz}B_p(\nabla\varphi)\,\dx=\int_{\BRz}B_p(\id)\,\dx
	\end{equation}
	for all radial deformations $\varphi(x)=v(\norm x)\.\frac{x}{\norm x}$ with $v\in\mathcal V_R$, where $\BRz$ denotes the ball with radius $R$ and center $0$.
\end{lemma}
\begin{proof}
	For $v\in\mathcal V_R$ we have $\lmax=\frac{v(\norm{x})}{\norm{x}}\geq v'(\norm{x})=\lmin\geq 0$ and calculate
	\begin{align}
		\int_{\BRz}-B_p(\nabla\varphi)\,\dx&=\int_{\BRz}\frac{p}{2}\.\frac{v'(\norm{x})\.v(\norm{x})^{p-1}}{\norm{x}^{p-1}}+\frac{2-p}{2}\.\frac{v(\norm{x})^p}{\norm{x}^p}\,\dx\notag\\
		&=2\pi\.\int_0^R\left(\frac{p}{2}\.\frac{v'(r)\.v(r)^{p-1}}{r^{p-1}}+\frac{2-p}{2}\.\frac{v(r)^p}{r^p}\right)r\,\dr\notag\\
		&=\pi\.\int_0^R p\.\frac{v'(r)\.v(r)^{p-1}}{r^{p-2}}+(2-p)\.\frac{v(r)^p}{r^{p-1}}\,\dr\\
		&=\pi\.\int_0^R\ddr\left[\frac{v(r)^p}{r^{p-2}}\right]\dr=\pi\left[\frac{v(R)^p}{R^{p-2}}-\lim_{r\to 0}\frac{v(r)^p}{r^{p-2}}\right]_0^R=\pi R^2\notag\\
		&=\int_{\BRz}1\,\dx=\int_{\BRz}-B_p(\id)\,\dx\,.\notag\qedhere
	\end{align}
\end{proof}
The following lemma establishes an analogous result for the function $\Wmp$ with
\begin{align}
	\Wmp(F)=\frac{\lmax}{\lmin}-\log\frac{\lmax}{\lmin}+\log(\lmax\lmin)=\frac{\lmax}{\lmin}+2\.\log\lmin\,.
\end{align}
\begin{lemma}\label{lemma:WmpRadial}
	The energy value of $\Wmp$ is constant on the class of \emph{contracting} functions $\mathcal V_R\inv$, i.e.\
	\begin{equation}
		\int_{\BRz}\Wmp(\nabla\varphi)\,\dx=\int_{\BRz}\Wmp(\id)\,\dx
	\end{equation}
	for all radial deformations $\varphi(x)=v(\norm x)\.\frac{x}{\norm x}$ with $v\in\mathcal V_R\inv$.
\end{lemma}
\begin{proof}
	We first extend Lemma \ref{lemma:burkholderRadial} to the function $B_\star\col\GLp(2)\to\R$ with
	\[
		B_\star(F)=\lim_{p\searrow 2}\frac{B_p(F)-B_2(F)}{p-2}=\ddp B_p(F)|_{p=2}=-\frac{1}{2}\left(1+\log\big(\opnorm{F}^2\big)\right)\det F+\frac{1}{2}\.\opnorm{F}^2,
	\]
	for all $F\in\GLp(2)$, where $B_2(F)=\det F$; note that $B_2$ is a Null-Lagrangian. For expanding deformations defined by $\mathcal V_R$ we find
	\begin{align*}
		\int_{\BRz}B_\star(\nabla\varphi)\,\dx&=\int_{\BRz}\lim_{p\searrow 2}\frac{B_p(\nabla\varphi)-B_2(\nabla\varphi)}{p-2}\,\dx=\lim_{p\searrow 2}\frac{1}{p-2}\int_{\BRz}B_p(\nabla\varphi)-B_2(\nabla\varphi)\,\dx\\
		&=\lim_{p\searrow 2}\frac{1}{p-2}\int_{\BRz}B_p(\id)-B_2(\id)\,\dx=\int_{\BRz}B_\star(\id)\,\dx\,,
	\end{align*}
	thus the energy potential of $B_\star$ is indeed constant on the class $\mathcal{V}_R$. Next, recall that $B_\star$ and $\Wmp$ are related via the Shield transformation, i.e.\ that $B_\star^\#(F)=\frac{1}{2}\bigl(\Wmp(F)-1\bigr)$ for all $F\in\GLp(2)$, and that for every expanding radial deformation $\varphi$, its inverse mapping $\varphi\inv$ is a contracting radial deformation, i.e.\ that for every $v(r)\in\mathcal{V}_R\inv$ there exists a unique $w(r)\in\mathcal{V}_R$ with $w=v\inv$. We can therefore apply the general formula
	\begin{equation}\label{eq:shieldGeneral}
		\int_\Omega W\big(\nabla\varphi(x)\big)\,\dx=\int_{\varphi(\Omega)} W\big(\big[\nabla_\xi(\varphi\inv)(\xi)\big]\inv\big)\det\nabla_\xi\varphi\inv(\xi)\,\intd\xi=\int_{\varphi(\Omega)} W^\#\big(\nabla_\xi(\varphi\inv)(\xi)\big)\,\intd\xi
	\end{equation}
	to $\varphi(\Omega)=\Omega=\BRz$ to find
	\begin{equation}
		\int_{\BRz}\Wmp(F_v)\,\dx=\int_{\BRz}2\.B_\star(F_w)+1\,\dx=\int_{\BRz}2\.B_\star(\id)+1\,\dx=\int_{\BRz}\Wmp(\id)\,\dx\,,
	\end{equation}
	where $F_v$ and $F_w$ are the deformation gradients corresponding to $v(r)$ and $w(r)=v\inv(r)$, respectively.
\end{proof}
Both the radial deformations $\varphi(x)=v(\norm x)\.\frac{x}{\norm x}$ defined by $v\in\mathcal V_R$ in Lemma \ref{lemma:burkholderRadial} and $v\in\mathcal V_R\inv$ in Lemma \ref{lemma:WmpRadial} can be extended by \cite{astala2012quasiconformal}
\begin{itemize}
	\item choosing any $F_0=\lambda\.Q$ instead of $\id$ for arbitrary $\lambda>0$ and $Q\in\SO(2)$,
	\item choosing any point $z_0\in\Omega$ instead of the origin,
	\item combining multiple separate balls $B_{R_i}(z_i)$ with $\bigcap_{i}B_{R_i}(z_i)=\emptyset$,
	\item inductively adding more balls $B_{r}(z_0')$ which are contained in a previous ball $B_{R}(z_0)$ under the assumption that the contracting deformation of the outer ball $B_{R}(z_0)$ maintains $F_0=\lambda\.Q$ at the boundary of the inner ball $B_{r}(z_0')$,
\end{itemize}
to construct an overall non-radially symmetric deformation (cf. Figure \ref{fig:circularPacking}) with the same energy level as the homogeneous deformation.

\begin{figure}[h!]
    \centering  
    \begin{tikzpicture}
     \begin{axis}[
        axis x line=middle,axis y line=middle,
        xmin=-1, xmax=1,
        ymin=-1, ymax=1,
        width=.51\linewidth,
        height=.53\linewidth,
        hide axis
        ]
        \addplot[very thick,domain=0:360,samples=\sample] ({0.99*cos(x)},{0.99*sin(x)});
		\addplot[thick,domain=0:360,samples=\sample] ({0.5*cos(x)+0.15},{0.5*sin(x)+0.25});
		\addplot[thick,domain=0:360,samples=\sample] ({0.3*cos(x)-0.6},{0.3*sin(x)-0.1});
		\addplot[thick,domain=0:360,samples=\sample] ({0.3*cos(x)},{0.3*sin(x)-0.6});
		\addplot[thick,domain=0:360,samples=\sample] ({0.2*cos(x)+0.6},{0.2*sin(x)-0.4});
		\addplot[thick,domain=0:360,samples=\sample] ({0.1*cos(x)-0.6},{0.1*sin(x)+0.5});
	\end{axis}
    \end{tikzpicture}
    \hfill
    \begin{tikzpicture}
     \begin{axis}[
        axis x line=middle,axis y line=middle,
        xmin=-1, xmax=1,
        ymin=-1, ymax=1,
        width=.51\linewidth,
        height=.53\linewidth,
        hide axis
        ]
        \addplot[very thick,domain=0:360,samples=\sample] ({0.99*cos(x)},{0.99*sin(x)});
		\addplot[very thick,domain=0:360,samples=\sample] ({0.99*cos(x)},{0.99*sin(x)});
		\addplot[thick,domain=0:360,samples=\sample] ({0.5*cos(x)+0.15},{0.5*sin(x)+0.25});
		\addplot[thick,domain=0:360,samples=\sample] ({0.3*cos(x)-0.6},{0.3*sin(x)-0.1});
		\addplot[thick,domain=0:360,samples=\sample] ({0.3*cos(x)},{0.3*sin(x)-0.6});
		\addplot[thick,domain=0:360,samples=\sample] ({0.2*cos(x)+0.6},{0.2*sin(x)-0.4});
		\addplot[thick,domain=0:360,samples=\sample] ({0.1*cos(x)-0.6},{0.1*sin(x)+0.5});
		\addplot[thick,domain=0:360,samples=\sample] ({0.35*cos(x)+0.15},{0.35*sin(x)+0.35});
		\addplot[thick,domain=0:360,samples=\sample] ({0.1*cos(x)-0.7},{0.1*sin(x)-0.2});
		\addplot[thick,domain=0:360,samples=\sample] ({0.1*cos(x)},{0.1*sin(x)-0.5});
		\addplot[thick,domain=0:360,samples=\sample] ({0.15*cos(x)+0.6},{0.15*sin(x)-0.4});
		\addplot[thick,domain=0:360,samples=\sample] ({0.15*cos(x)-0.5},{0.15*sin(x)-0.6});
		\addplot[thick,domain=0:360,samples=\sample] ({0.15*cos(x)-0.52},{0.15*sin(x)});
		\addplot[thick,domain=0:360,samples=\sample] ({0.1*cos(x)+0.1},{0.1*sin(x)-0.7});
		\addplot[thick,domain=0:360,samples=\sample] ({0.15*cos(x)+0.8},{0.15*sin(x)});
		\addplot[thick,domain=0:360,samples=\sample] ({0.05*cos(x)-0.3},{0.05*sin(x)+0.7});
		\addplot[thick,domain=0:360,samples=\sample] ({0.05*cos(x)-0.1},{0.05*sin(x)});
		\addplot[thick,domain=0:360,samples=\sample] ({0.05*cos(x)+0.1},{0.05*sin(x)-0.1});
		\addplot[thick,domain=0:360,samples=\sample] ({0.05*cos(x)+0.35},{0.05*sin(x)-0.05});
		\addplot[thick,domain=0:360,samples=\sample] ({0.05*cos(x)-0.15},{0.05*sin(x)-0.7});
		\addplot[thick,domain=0:360,samples=\sample] ({0.05*cos(x)-0.8},{0.05*sin(x)+0.3});
		\addplot[thick,domain=0:360,samples=\sample] ({0.05*cos(x)+0.4},{0.05*sin(x)-0.7});
		\addplot[thick,domain=0:360,samples=\sample] ({0.05*cos(x)+0.75},{0.05*sin(x)+0.5});
	\end{axis}
    \end{tikzpicture}
    \caption{Circular packing of several balls with various contracting mappings to achieve an overall non-symmetric deformation, also called \emph{piecewise radial mapping}.\label{fig:circularPacking}}
\end{figure}
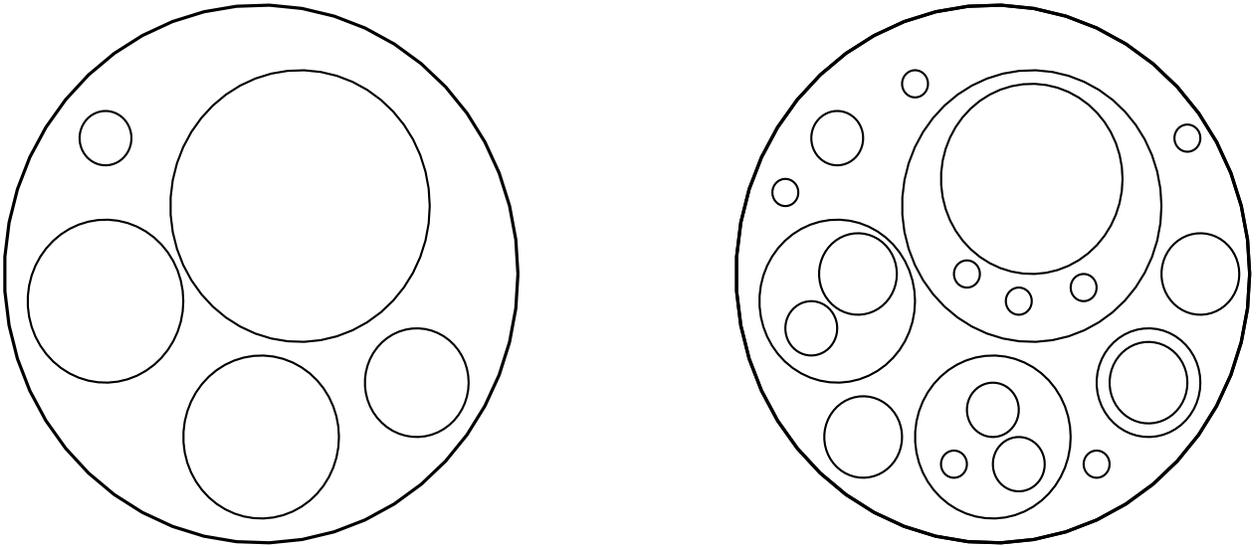
%
%
%
%
%
\section{Summary}
The class of planar energy functions with an additive volumetric-isochoric split, i.e.\ functions on $\GLp(2)$ of the form
\begin{equation}
	W(F)=h\bigg(\frac{\lambda_1}{\lambda_2}\bigg)+f(\lambda_1\lambda_2)\,,\qquad h,f\col\Rp\to\R\,,\quad h(t)=h\bigg(\frac{1}{t}\bigg)\qquad\text{for all }\;t\in(0,\infty)\,,
\end{equation}
where $\lambda_1,\lambda_2>0$ denote the singular values of $F$ and $h,f$ are given real-valued functions, proves to be a rather interesting subject in the context of Morrey's conjecture: On the one hand, if neither $h$ nor $f$ is convex, then $W$ is not rank-one convex; on the other hand, convexity of both $f$ and $h$ already implies that $W$ is quasiconvex.

Investigating the remaining cases, we first focused on functions with a convex isochoric part $h$. In that case, Lemma \ref{lem:minimallyConvexEnergy} shows that the question whether or not rank-one convexity implies quasiconvexity can be reduced to the question whether the particular (rank-one convex) function $\Wmp\col\GLp(2)\to\R$ with
\begin{equation}
	\Wmp(F)=\frac{\lambdamax}{\lambdamin} - \log\left(\frac{\lambdamax}{\lambdamin}\right) + \log\det F=\frac{\lmax}{\lmin}+2\.\log\lmin
\end{equation}
is quasiconvex. While $\Wmp$ is not polyconvex according to Lemma \ref{lem:WmpNotPolyconvex}, determining whether or not $\Wmp$ is indeed quasiconvex remains an open problem at this point.

The case of volumetric-isochorically split functions with a convex volumetric part $f$ appears to be more involved: In Section \ref{sec:Mm}, we identified three \enquote{least} rank-one convex functions $\Wmm(F),\,\Ws(F),$ and $\Wn(F)$, none of which is currently known to be quasiconvex. A comprehensive overview for candidates with an additive volumetric-isochoric split is given in Figure \ref{fig:morreySummary}.

\subsection{Symmetry and stability}

The energy potential $\Wmp$ also exhibits a number of interesting properties regarding the relation between symmetry and stability: According to Lemma \ref{lemma:WmpRadial}, $\Wmp$ is constant for the non-trivial class of contracting radially symmetric deformations, i.e.\ $\Wmp$ allows for inhomogeneous deformations whose energy level is equal to the homogeneous one. This implies that the homogeneous solution to the energy minimization problem is neither unique nor stable.

As shown in Figure \ref{fig:circularPacking}, it is thereby possible to construct a non-symmetrical microstructure for the fully isotropic energy potential $\Wmp$ even if the applied boundary conditions are fully symmetrical.
This is in stark contrast to the linear convex case, where such symmetrically posed problems generally yield symmetric solutions; in practice, this observation allows for certain simplifications (e.g.\ by reducing calculations to any quarter of a cylinder) of symmetric linear problems. However, in the nonlinear non-convex case, it is even possible for a symmetrically stated problem (cf.\ equation \eqref{eq:nonRadialMinimizer}) to not have a stable symmetric solution at all \cite{brock1996symmetry}. This phenomenon, which is exhibited by the energy function $\Wmp$ as well, demonstrates a \enquote{preference of disorder over order} which can occur in non-(strictly-)quasiconvex problems.

Note that due to Lemma \ref{lem:radialSeparateConvex}, radially symmetric deformations cannot be energetically favorable compared to the affine solution even under the weaker assumption of separate convexity. Additionally, in an intriguing article by Kawohl \cite{kawohl1988quasiconvexity}, it is shown that the simple assumption $\varphi(x,y)=\varphi(-x,y)$ for the test function $\varphi\in W_0^{1,\infty}(\Omega)$ on the domain $\Omega=(-1,1)^2$ is enough for the transition from rank-one convexity to quasiconvexity. Thus in general, the break of symmetry plays an essential role in stability and, in particular, must be taken into account for the minimization of potentially non-quasiconvex energies.
%
%
%
%
%
%
%
%
%
%
\section{References}
\footnotesize
\printbibliography[heading=none]
%
%
%
%
%
\small
\begin{appendix}
\section{Complex Analysis}\label{sec:complexAnalysis}
As indicated in Section \ref{sec:Burkholder}, the energy function $\Wmp$ has an intriguing connection to the work of Iwaniec in the field of complex analysis and Burkholder's martingale theory. Here, we briefly discuss the conversion of the energy $\Wmp(F)$ into complex notation.

\newcommand{\fc}{\mathfrak{f}}
Generally, we now consider a complex-valued function $\fc\col U\subset\C\to\C$ with $z\mapsto \fc(z)$. (Note that this function $\fc$ should not be  mistaken with the volumetric part $f\col\Rp\to\R$ used in the additive volumetric-isochoric split.) 
It is common to use the complex vector $(\partial_z\fc,\partial_\zbar \fc)$ to describe the derivative of a non-holomorphic function on $\C$ instead of the deformation gradient $F=D\fc(z)\in\R^{2\times2}$. Thus we need to rephrase any energy function on $\R^{2\times2}$ or on $\GLp(2)$ as a mapping $L\col\C\times\C\to\R$ and consider the corresponding minimization problem
\begin{equation}
	\min_{\fc\col U\to\C}\int_U L(\partial_z\fc,\partial_\zbar \fc)\,\dz\,.
\end{equation}
On the other hand, we can always translate $(z,w)\mapsto L(z,w)$ to a real-valued function $W\col\R^{2\times 2}\to\R$ using the identity
\begin{align}
		\partial_z\fc(z)&=\frac{1}{2}\left(\partial_x-i\.\partial_y\right)\left[u(x,y)+i\.v(x,y)\right]=\frac{1}{2}\left[\partial_xu+\partial_yv+i\left(\partial_xv-\partial_yu\right)\right]\,,\\
		\partial_\zbar \fc(z)&=\frac{1}{2}\left(\partial_x+i\.\partial_y\right)\left[u(x,y)+i\.v(x,y)\right]=\frac{1}{2}\left[\partial_xu-\partial_yv+i\left(\partial_xv+\partial_yu\right)\right]
\end{align}
to compute the corresponding deformation gradient $F=D\fc(z)$, i.e.\ using the isomorphism
\begin{align}
	\bigl(z(F),w(F)\bigr)=\matr{F_{11}+F_{22} & F_{11}-F_{22}\\ i\.(F_{21}-F_{12}) & i\.(F_{21}+F_{12})}\quad\iff\quad F(z,w)=\matr{\Re z+\Re w & \Im w -\Im z\\ \Im z+\Im w & \Re z-\Re w}.\label{eq:complexIsomorphism}
\end{align}
\begin{example}
	We consider $L\col\C\times\C\to\R$ with
	\begin{equation}
		L(z,w)=\abs z^2-\abs w^2
	\end{equation}
	and compute
	\begin{align}
		L(\partial_z\fc,\partial_\zbar \fc)&=\abs{\partial_z\fc}^2-\abs{\partial_\zbar \fc}^2\notag\\
		&=\left|\frac{1}{2}\left[\partial_xu+\partial_yv+i\left(\partial_xv-\partial_yu\right)\right]\right|^2-\left|\frac{1}{2}\left[\partial_xu-\partial_yv+i\left(\partial_xv+\partial_yu\right)\right]\right|^2\notag\\
		&=\frac{1}{4}\left[\left(\partial_xu+\partial_yv\right)^2+\left(\partial_xv-\partial_yu\right)^2-\left(\partial_xu-\partial_yv\right)^2-\left(\partial_xv+\partial_yu\right)^2\right]\notag\\
		&=\frac{1}{4}\left[\partial_xu^2+2\.\partial_xu\.\partial_yv+\partial_yv^2+\partial_xv^2-2\.\partial_xv\.\partial_yu+\partial_yu^2\right.\label{eq:complexDeterminant}\\
		&\phantom{=}\;\left.-\partial_xu^2+2\.\partial_xu\.\partial_yv-\partial_yv^2-\partial_xv^2-2\.\partial_xv\.\partial_yu-\partial_yu^2\right]\notag\\
		&=\frac{1}{2}\left[\partial_xu\.\partial_yv-\partial_xv\.\partial_yu+\partial_xu\.\partial_yv-\partial_xv\.\partial_yu\right]\notag\\
		&=\partial_xu\.\partial_yv-\partial_xv\.\partial_yu=J_\fc(z)\,.\notag
	\end{align}
	Thus we can identify the complex mapping $L(z,w)=\abs z^2-\abs w^2$ with the isotropic energy function
	\begin{equation}
		W(F)=F_{11}F_{22}-F_{12}F_{21}=\det F\,.
	\end{equation}
	Similarly, since
	\begin{align}
		\abs{\partial_z\fc}^2+\abs{\partial_\zbar \fc}^2
		&=\left|\frac{1}{2}\left[\partial_xu+\partial_yv+i\left(\partial_xv-\partial_yu\right)\right]\right|^2+\left|\frac{1}{2}\left[\partial_xu-\partial_yv+i\left(\partial_xv+\partial_yu\right)\right]\right|^2\notag\\
		&=\frac{1}{4}\left[\left(\partial_xu+\partial_yv\right)^2+\left(\partial_xv-\partial_yu\right)^2+\left(\partial_xu-\partial_yv\right)^2+\left(\partial_xv+\partial_yu\right)^2\right]\notag\\
		&=\frac{1}{4}\left[\partial_xu^2+2\.\partial_xu\.\partial_yv+\partial_yv^2+\partial_xv^2-2\.\partial_xv\.\partial_yu+\partial_yu^2\right.\label{eq:complexNorm}\\
		&\phantom{=}\;\left.+\partial_xu^2-2\.\partial_xu\.\partial_yv+\partial_yv^2+\partial_xv^2+2\.\partial_xv\.\partial_yu+\partial_yu^2\right]\notag\\
		&=\frac{1}{2}\left[\partial_xu^2+\partial_yv^2+\partial_xv^2+\partial_yu^2\right]=\partial_xu^2+\partial_yv^2,\notag
	\end{align}
	we identify $L(z,w)=\abs z^2+\abs w^2$ with $W(F)=\frac{1}{2}\.\norm F^2$.
\end{example}
Analogous computations yield the representations
\begin{align}
	\abs z^2&=\frac{1}{2}\left[\left(\abs z^2+\abs w^2\right)+\left(\abs z^2-\abs w^2\right)\right]=\frac{1}{4}\.\norm F^2+\frac{1}{2}\.\det F\,,\notag\\
	\abs w^2&=\frac{1}{2}\left[\left(\abs z^2+\abs w^2\right)-\left(\abs z^2-\abs w^2\right)\right]=\frac{1}{4}\.\norm F^2-\frac{1}{2}\.\det F\,,\notag\\
	\abs z\abs w&=\sqrt{\frac{1}{4}\.\norm F^2+\frac{1}{2}\.\det F}\cdot\sqrt{\frac{1}{4}\.\norm F^2-\frac{1}{2}\.\det F}=\frac{1}{4}\.\sqrt{\norm F^4-4\.(\det F)^2}\notag\,,\\
	\left(\abs z+\abs w\right)^2 &= \abs z^2+\abs w^2+2\.\abs z\abs w = \frac{1}{2}\.\norm F^2+\frac{2}{4}\sqrt{\norm{F}^4-4\.(\det F)^2}\notag\\
	&= \frac{\norm{F}^2+\sqrt{\norm{F}^4-4\.(\det F)^2}}{2} = \opnorm F^2 = \lmax^2\,,\label{eq:lmaxComplex}\\
	\abs z -\abs w &= \frac{\abs z^2-\abs w^2}{\abs z+\abs w}=\frac{\det F}{\lmax}=\lmin\,,\notag\\
	\K(F) &= \frac{\lmax}{\lmin} = \frac{\abs z +\abs w}{\abs z -\abs w}\,,\notag\\
	K(F) &= \frac{\norm F^2}{2\det F} = \frac{\abs z^2+\abs w^2}{\abs z^2-\abs w^2}\,,\notag
\end{align}	
which can be used to to convert the energy
\begin{equation}
	\Wmp(F)=\frac{\lmax}{\lmin} - \log\frac{\lmax}{\lmin} + \log\det F = K(F)+2\.\log\lmin = \frac{\abs z +\abs w}{\abs z -\abs w}+2\.\log\left(\abs z-\abs w\right)
\end{equation}
to a mapping $L\col\C\times\C\to\R$; note that $\abs z>\abs w$ holds for arbitrary $F\in\GLp(2)$ because\footnote{Similarly, $\det F=\abs z-\abs w=0$ for any rank-one matrix $F=\xi\otimes\eta$ with $\xi,\eta\in\R^2$.} $\abs z-\abs w=\det F>0$.

As indicated in Section \ref{sec:Burkholder}, the Burkholder functional\footnote{Usually, for the Burkholder functional one refers to a work by Burkholder about sharp estimates for martingales \cite{burkholder1988sharp}. However, there is some disagreement about which expression exactly defines the function. Here, we use the expression $B_p(F)$ as defined by Astala et al.\ \cite{astala2012quasiconformal} but a reversed sign so that we are looking for quasiconvexity instead of quasiconcavity. Baernstein \cite{baernstein2011some}, for example, generalizes the Burkholder function to include the case $p\in(1,2)$ as well and introduces an additional constant $\alpha_p$ (cf.\ inequality \eqref{eq:BurkholderAlternative}).}
\begin{equation}
	 B_p(F)=-\left[\frac{p}{2}\.\det F+\left(1-\frac{p}{2}\right)\opnorm{F}^2\right]\opnorm{F}^{p-2}
	 \,,
\end{equation}
with $p\geq 2$ and
\begin{equation}
	\opnorm F =\lmax = \sqrt{\frac{\norm{F}^2+\sqrt{\norm{F}^4-4\.(\det F)^2}}{2}}
\end{equation}
denoting the operator norm, plays an important role in complex analysis. Its representation as a complex mapping $L_p\col\C\times\C\to\R$ can be computed via
\begin{align}
	-B_p(F)&=\left[\frac{p}{2}\.\det F+\left(1-\frac{p}{2}\right)\opnorm{F}^2\right]\opnorm{F}^{p-2}\notag\\
	&=\left[\frac{p}{2}\left(\abs z^2-\abs w^2\right)+\left(1-\frac{p}{2}\right)\left(\abs z +\abs w\right)^2\right]\left(\abs z +\abs w\right)^{p-2}\notag\\
	&=\frac{1}{2}\left[p\.\abs z^2-p\.\abs w^2+(2-p)\left(\abs z^2+2\.\abs z\abs w +\abs w^2\right)\right]\left(\abs z +\abs w\right)^{p-2}\\
	&=\left[\abs z^2+(2-p)\.\abs z\abs w+(1-p)\abs w^2\right]\left(\abs z +\abs w\right)^{p-2}\notag\\
	&=\left(\abs z-(p-1)\abs w\right)\left(\abs z +\abs w\right)^{p-1}\equalscolon L_p(z,w)\notag
	\,.
\end{align}
 
In addition to its connection to Morrey's conjecture (cf.\ Section \ref{sec:Burkholder}), the Burkholder functional $B_p(F)$, or rather its complex representation $L_p(z,w)$, plays an important role in martingale inequalities and the \emph{Beurling-Ahlfors transform} (also called \emph{complex Hilbert transform}) $S\col L^p(\C)\to L^p(\C)$, defined as the singular integral
\begin{equation}
	(S\omega)(z)=-\frac1\pi\iint_\C\frac{\omega(\xi)}{(z-\xi)^2}\,\intd\xi\,,\qquad\omega\in L^p(\C)\,.
\end{equation}
One primary aim in function theory is determining the $p$-norm
\[
	\norm{S}_{p}=\sup_{\omega\in L^p(\C)}\frac{\norm{S(\omega)}_p}{\norm{\omega}_p}\,,\qquad 1<p<\infty
\]
of the Beurling transform. A yet unsolved conjecture of Iwaniec \cite{iwaniec1982extremal,iwaniec1993quasiregular,iwaniec2002nonlinear} states that
\begin{equation}
	\norm{S}_{p} = p^\star - 1 = \begin{cases}p-1 &\casesif 2\leq p<\infty\,,\\\frac{1}{p-1} &\casesif 1<p\leq 2\,.\end{cases}\label{eq:IwaniecConjecture}
\end{equation}
We can shorten the notation in \eqref{eq:IwaniecConjecture} by introducting $p^\star=\max(p,p')$ with $\frac{1}{p}+\frac{1}{p'}=1$ for arbitrary $p>1$. It is already known \cite{iwaniec1993quasiregular} that $\norm{S}_{p}\geq p^\star - 1$ as well as $\norm{S}_{p}\leq C_n\.(p^\star - 1)$, where the constant $C_n$ is at best exponential in $n$. 
Burkholder \cite{burkholder1988sharp} proved that the inequality
\begin{equation}
	\abs z^p -(p^\star-1)^p\abs w^p \leq \alpha_p\left(\abs z-(p^\star-1)\.\abs w\right)\left(\abs z +\abs w\right)^{p-1},\qquad\alpha_p=p\left(1-\frac{1}{p^\star}\right)^{p-1},\label{eq:BurkholderAlternative}
\end{equation}
holds for all $p>1$. In the case $p\geq 2$, we can identify Burkholder's functional $L_p(z,w)$ in complex notion with the right hand side of \eqref{eq:BurkholderAlternative}. The term on the left-hand side,
\begin{equation}
	V(z,w) = \abs z^p -(p^\star-1)^p\abs w^p\label{eq:BurkholderV}
	\,,
\end{equation}
plays an important role in sharp estimates for martingales \cite{burkholder1988sharp}. To prove Iwaniec's conjecture \eqref{eq:IwaniecConjecture}, the expression \cite{astala2015hunt}
\begin{equation}
	\norm{S(\fc)}_p\leq(p^\star-1)\norm{\fc}_p\,,\qquad 1<p<\infty
\end{equation}
must be shown to hold for all $\fc\in L_p(\C)$. With the identity (cf.\ \cite{banuelos1997martingale}) $\displaystyle S\circ\frac{\partial}{\partial\zbar}=\frac{\partial}{\partial z}$, the above inequality is equivalent to
\begin{equation}
	\norm{\partial_z\fc}_p\leq(p^\star-1)\norm{\partial_\zbar \fc}_p\,,\qquad \fc\in C_0^\infty(\C)\,,
\end{equation}
which shows a direct connection to the mapping $V\col\C\times\C\to\R$ in equation \eqref{eq:BurkholderV}. The conjecture of Iwaniec would follow \cite{banuelos1997martingale} from
\begin{equation}
	\iint_\C L_p(\partial_z \fc,\partial_\zbar \fc)\,\intd\xi\leq 0\qquad \forall\;\fc\in C_0^\infty(\C)\,.\label{eq:BurkholderIwaniecConjecture}
\end{equation}
This inequality is the motivatition for the claim of quasiconcavity of $L_p(z,w)$, i.e.\ quasiconvexity of the Burkholder functional $B_p(F)$, as a sufficient but not neccesarry statement to ensure Iwaniec's conjecture \cite{davis2011donald}. In fact, inequality \eqref{eq:BurkholderIwaniecConjecture} is equivalent to quasiconcavity of $L_p$ at 0 \cite{banuelos2010foundational}. In an older work of Baernstein \cite{baernstein2011some} they present some first numerical evidence in favor of inequality \eqref{eq:BurkholderIwaniecConjecture} up to machine precision, using a $2N^2$ real dimensional space splitting the unit square $[0,1]^2$ into triangles to approximate $L_P(z,w)$ by piecewise linear functions (with N ranging from 6 to 100).
\end{appendix}
\end{document}